\documentclass[11pt]{article}
\usepackage{graphicx}

\usepackage{amsmath,amssymb,amsthm}
\usepackage{color}
\usepackage{comment}
\usepackage{float}
\usepackage[a4paper,margin=2.5cm]{geometry}

\newtheorem{thm}{Theorem}[section]

\newtheorem{lem}[thm]{Lemma}

\newtheorem{con}[thm]{Conjecture}

\theoremstyle{definition}

\newtheorem{example}{Example}[section]

\theoremstyle{remark}

\numberwithin{equation}{section}

\title{On the Ordering and Injectivity of Lagrange Constants Associated with Certain Substitutions}

\author{Shin-ichi Yasutomi}

\begin{document}

\maketitle


\begin{abstract}
Motivated by the modern characterization of the Markoff spectrum via mechanical words and the Markoff Uniqueness Conjecture, 
we study the Lagrange constants of continued fractions generated by a parameterized family of substitutions $\phi_{a,b}$
defined by $0 \mapsto aa$ and $1 \mapsto bb$ with $b = a + 1$. Based on a three-fold arithmetic classification of pairs of rational numbers $0 \le x < y \le 1$, we investigate the order relations between the corresponding Lagrange constants $\mathcal{L}([\phi_{a,a+1}(G(x))])$ associated with the mechanical words $G(x)$. While an inequality is established for pairs of the second type (type (2)) whenever $a \ge 1$, the comparisons for type (1) and type (3) pairs hold for $a \ge 2$. Consequently, for any integer $a \ge 2$, where the order relations across all three types are completely determined, we establish an analogue of the Markoff Uniqueness Conjecture: the map $x \mapsto \mathcal{L}([\phi_{a,a+1}(G(x))])$ is injective on $\mathbb{Q} \cap [0, 1]$. Furthermore, supported by extensive numerical computations, we propose a general conjecture on the order relations for arbitrary $b > a$, highlighting a sharp phase transition in ordering behavior at the boundary $b = a^2 + 2$.
\end{abstract}

\section{Introduction}
Let $\alpha$ be an irrational number. The Lagrange constant of $\alpha$, denoted by $\mathcal{L}(\alpha)$, is defined as
\begin{align}
\mathcal{L}(\alpha) = \limsup_{q \to \infty} \frac{1}{q \|q\alpha\|},
\end{align}
where $q$ ranges over positive integers, and $\|x\|$ denotes the distance from $x$ to the nearest integer.
The set $\mathcal{L}$ is called the Lagrange spectrum, defined as:
\begin{align}
\mathcal{L} = \{ \mathcal{L}(\alpha) \mid \alpha \in \mathbb{R} \setminus \mathbb{Q} \}.
\end{align}
The minimum value of $\mathcal{L}$ is $\sqrt{5}$. In the range below $3$, 
A.~Markov \cite{Markoff1879,Markoff1880} established that $\mathcal{L}$ forms a discrete sequence that accumulates to $3$.
Namely, a positive integer $m \in \mathbb{Z}_{>0}$ is called a Markov number
if it is part of a positive integer solution $(x,y,z)$ to $x^2 + y^2 + z^2 = 3xyz$. In this notation,
\begin{align}
\mathcal{L} \cap (-\infty, 3) = \left\{ \sqrt{9 - \frac{4}{m^2}} \;\middle|\; m \text{ is a Markov number} \right\},
\end{align}
which is denoted by $\mathcal M$.
It is known that the continued fraction expansion of a number whose Lagrange constant belongs to $\mathcal{M}$ can be
 expressed in terms of positive rational numbers less than or equal to $1$.
To show this, let us introduce some notation. 
For a rational number $x$ satisfying $0 \le x \le 1$, we consider the right-infinite word $G(x)$ generated by the alphabet $\{0, 1\}$ defined as follows:
\begin{align}
G(x) = G(x,1)G(x,2)\ldots G(x,n)\ldots,
\end{align}
where
\begin{align}
G(x,n) = \lfloor nx \rfloor - \lfloor (n-1)x \rfloor.
\end{align}
For example, $G(0) = 00\dots = 0^\infty$ and $G(1/2) = 0101\dots = (01)^\infty$.
Let $\phi$ be the substitution defined by $0 \to 11$ and $1 \to 22$.
For a finite or right-infinite word $w = w_1 w_2 \dots$ over an alphabet $A \subset \mathbb{Z}_{>0}$, we denote by $[w]$ the regular continued fraction expansion
\begin{equation}
[w] = [0; w_1, w_2, \dots].
\end{equation}
The following characterization is a modern formulation of the classical theorem of Markoff, expressed in terms of Christoffel (or Sturmian) words; see Reutenauer \cite{Reutenauer2019}.
\begin{thm}\label{t1}
Let $\alpha \in \mathbb{R} \setminus \mathbb{Q}$.
Then $\mathcal{L}(\alpha) \in \mathcal{M}$ if and only if there exists a rational number $x$ satisfying $0 \le x \le 1$ such that $\alpha \sim [\phi(G(x))]$.
Here, $\alpha \sim \beta$ means that the continued fraction expansions of $\alpha$ and $\beta$ eventually coincide.
\end{thm}

Using the characterization of Theorem \ref{t1}, the Markoff Uniqueness Conjecture can be stated as follows.

\begin{con}
Let $x, y \in \mathbb{Q}$ such that $0 \le x, y \le 1$.
If $\mathcal{L}([\phi(G(x))]) = \mathcal{L}([\phi(G(y))])$, then $x = y$.
\end{con}

Motivated by the characterization in Theorem \ref{t1}, we study continued fractions obtained by replacing the substitution $0 \to 11$ and $1 \to 22$
 with more general substitutions. We then define the substitutions $\phi_{a,b}$ and investigate the associated Lagrange constants.

Let $a, b \in \mathbb{N}$.
Let $\phi_{a,b}$ be the substitution defined by $0 \mapsto aa$ and $1 \mapsto bb$.
Then we have the following theorem.

\begin{thm}\label{mt1}
Let $a \in \mathbb{Z}_{>0}$ with $a > 1$, and let $b = a + 1$.
Let $x, y \in \mathbb{Q}$ such that $0 \le x, y \le 1$.
If $\mathcal{L}([\phi_{a,b}(G(x))]) = \mathcal{L}([\phi_{a,b}(G(y))])$, then $x = y$.
\end{thm}

Theorem \ref{mt1} follows from a more detailed result.
It becomes possible to describe the order relations among Lagrange constants.
Although the case $a=1, b=2$ corresponds to the classical case, it allows us to describe the order relations among some of the Lagrange constants.
To state the result, 
we prepare the following lemma.The proof will be given in the next section.
\begin{lem}\label{l1}
Let $x, y \in \mathbb{Q}$ such that $0 \le x < y \le 1$.
Let $\mathrm{den}(u)$ denote the denominator of a rational number $u$.
Then, the relation between $x$ and $y$ is classified into the following three types:
\begin{enumerate}
\item[\rm (1)] There exists a natural number $n$ such that $n < \min(\mathrm{den}(x), \mathrm{den}(y))$ and $\lfloor nx \rfloor < \lfloor ny \rfloor$.
\item[\rm (2)] $(x, y) = (0, 1)$, or $\mathrm{den}(x) < \mathrm{den}(y)$, and $\lfloor jx \rfloor = \lfloor jy \rfloor$ for all $1 \le j \le \mathrm{den}(x)$.
\item[\rm (3)] $\mathrm{den}(x) > \mathrm{den}(y)$, $\lfloor jx \rfloor = \lfloor jy \rfloor$ for all $1 \le j < \mathrm{den}(y)$, and $\lfloor jx \rfloor < \lfloor jy \rfloor$ for $j = \mathrm{den}(y)$.
\end{enumerate}
\end{lem}

 For $x, y \in \mathbb{Q}$ satisfying $0 \le x < y \le 1$,
  if they satisfy the $n$-th relation in Lemma~\ref{l1}, we say that $x$ and $y$ are of type $(n)$ (or that $x$ and $y$ have type $(n)$).

Our main theorem is as follows.

\begin{thm}\label{mt2}
Let $x, y \in \mathbb{Q}$ satisfy $0 \le x < y \le 1$.
\begin{enumerate}
\item[\rm (1)] If $x$ and $y$ are of type (1) and $a \in \mathbb{Z}_{>0}$ satisfies $a \ge 2$, then
\[
\mathcal{L}([\phi_{a,a+1}(G(x))]) < \mathcal{L}([\phi_{a,a+1}(G(y))]).
\]
\item[\rm (2)] If $x$ and $y$ are of type (2) and $a \in \mathbb{Z}_{>0}$ satisfies $a \ge 1$, then
\[
\mathcal{L}([\phi_{a,a+1}(G(x))]) < \mathcal{L}([\phi_{a,a+1}(G(y))]).
\]
\item[\rm (3)] If $x$ and $y$ are of type (3) and $a \in \mathbb{Z}_{>0}$ satisfies $a \ge 2$, then
\[
\mathcal{L}([\phi_{a,a+1}(G(x))]) > \mathcal{L}([\phi_{a,a+1}(G(y))]).
\]
\end{enumerate}
\end{thm}

\begin{example}
We illustrate each type in Lemma~\ref{l1} and Theorem~\ref{mt2} with explicit numerical examples. In particular, we highlight the classical case $a = 1$ corresponding to the original Markoff substitution $\phi_{1,2}$:
\begin{enumerate}
\item[\rm (1)] Consider $x = 1/3$ and $y = 3/4$, so that $x < y$. The pair $(x, y)$ is of type (1). 
Indeed, taking $n = 2$, we have $n < \min(\mathrm{den}(x), \mathrm{den}(y)) = \min(3, 4)$ and
\[
\lfloor 2 \cdot (1/3) \rfloor = 0 < 1 = \lfloor 2 \cdot (3/4) \rfloor.
\]
For $a = 2$, Theorem~\ref{mt2} (1) yields
\[
\mathcal{L}([\phi_{2,3}(G(1/3))]) < \mathcal{L}([\phi_{2,3}(G(3/4))]).
\]

\item[\rm (2)] Consider $x = 1/2$ and $y = 2/3$. The pair $(x, y)$ is of type (2). 
Indeed, we have $\mathrm{den}(x) = 2 < 3 = \mathrm{den}(y)$, and for $1 \le j \le 2$,
\[
\lfloor 1 \cdot (1/2) \rfloor = \lfloor 1 \cdot (2/3) \rfloor = 0 \quad \text{and} \quad \lfloor 2 \cdot (1/2) \rfloor = \lfloor 2 \cdot (2/3) \rfloor = 1.
\]
Since Theorem~\ref{mt2} (2) holds for all $a \ge 1$, we may take the original Markoff substitution $\phi_{1,2}$ ($a = 1$), which gives
\[
\mathcal{L}([\phi_{1,2}(G(1/2))]) < \mathcal{L}([\phi_{1,2}(G(2/3))]).
\]

\item[\rm (3)] Consider $x = 1/3$ and $y = 1/2$. The pair $(x, y)$ is of type (3). 
Indeed, we have $\mathrm{den}(x) = 3 > 2 = \mathrm{den}(y)$, and
\[
\lfloor 1 \cdot (1/3) \rfloor = \lfloor 1 \cdot (1/2) \rfloor = 0, \quad \text{while} \quad \lfloor 2 \cdot (1/3) \rfloor = 0 < 1 = \lfloor 2 \cdot (1/2) \rfloor.
\]
For $a = 2$, Theorem~\ref{mt2} (3) yields the reversed inequality
\[
\mathcal{L}([\phi_{2,3}(G(1/3))]) > \mathcal{L}([\phi_{2,3}(G(1/2))]).
\]
\end{enumerate}
\end{example}

Several ordering properties of Markoff numbers indexed by rational numbers have been studied in the literature (see, e.g., \cite{Aigner2013, Lee2023,Rabideau2020}). In contrast, our approach relies on a combinatorial order on the associated mechanical words.
While various generalizations of Markoff numbers and Markoff-type equations have been actively studied from algebraic and cluster-algebraic perspectives (see, e.g., \cite{Banaian2026,Gyoda2026}), our approach takes a different, combinatorial route: we generalize the continued-fraction construction through mechanical words and study the resulting Lagrange constants directly.

The paper is organized as follows. In Section~\ref{sec:prelim}, we prove Lemma~\ref{l1} and recall basic properties of the sequence $\{\lfloor nx \rfloor - \lfloor (n-1)x \rfloor\}$ for rational $x$. In Section~\ref{sec:lagrange}, we establish explicit formulas for calculating the Lagrange constants of $[\phi_{a,b}(G(x))]$. Section~\ref{sec:main} is devoted to the proof of our main theorems, which is divided into three parts according to the arithmetic classification. Finally, in Section~\ref{sec:conjecture}, supported by extensive numerical evidence, we propose a general conjecture regarding the ordering of Lagrange constants for arbitrary $b > a$, 
suggesting $b=a^2+2$ as a threshold for a change in the ordering behavior.

\section{Rational Mechanical Sequences}\label{sec:prelim}
Mechanical sequences provide an arithmetic description of binary sequences associated with rotations. We recall here some basic properties of mechanical sequences, for example, \cite{Lothaire2002}.
In this section, we establish several lemmas concerning rational slopes that will be used in our study.
First, we give the proof of Lemma \ref{l1}.
For $u \in \mathbb{Q}$, let $\mathrm{den}(u)$ and $\mathrm{num}(u)$ denote its denominator and numerator, respectively.
Note that we assume $\mathrm{den}(u) \ge 1$.

\begin{proof}[Proof of Lemma \ref{l1}]
First, assume that $(x,y)\ne (0,1)$ and $\mathrm{den}(x) = \mathrm{den}(y) = K$.
Then, we have
\[
\lfloor Kx \rfloor = \mathrm{num}(x) < \mathrm{num}(y) = \lfloor Ky \rfloor,
\]
which implies that
\[
\lfloor (K-1)x \rfloor = \lfloor Kx \rfloor - 1 < \lfloor Ky \rfloor - 1 = \lfloor (K-1)y \rfloor.
\]
Thus, $x$ and $y$ satisfy relation (1).

Next, assume that $\mathrm{den}(x) \ne \mathrm{den}(y)$.
Since $x < y$, there exists a natural number $m$ such that $\lfloor mx \rfloor < \lfloor my \rfloor$; in particular, let $m$ be the minimal such number.
If $m < \min(\mathrm{den}(x), \mathrm{den}(y))$, then $x$ and $y$ satisfy relation (1).
Now, suppose that $m \ge \min(\mathrm{den}(x), \mathrm{den}(y))$.
Assume further that $\mathrm{den}(x) < \mathrm{den}(y)$.
By the minimality of $m$, we have
\[
\lfloor jx \rfloor = \lfloor jy \rfloor \quad \text{for all } 1 \le j < \mathrm{den}(x).
\]
In particular, $\lfloor (\mathrm{den}(x)-1)x \rfloor = \lfloor (\mathrm{den}(x)-1)y \rfloor$.
Since $x < y$, we obviously have $\lfloor \mathrm{den}(x)x \rfloor \le \lfloor \mathrm{den}(x)y \rfloor$.
On the other hand,
\[
\lfloor \mathrm{den}(x)y \rfloor \le \lfloor (\mathrm{den}(x)-1)y \rfloor + 1 = \lfloor (\mathrm{den}(x)-1)x \rfloor + 1 = \lfloor \mathrm{den}(x)x \rfloor.
\]
Therefore, we obtain $\lfloor \mathrm{den}(x)x \rfloor = \lfloor \mathrm{den}(x)y \rfloor$.
Thus, $x$ and $y$ satisfy relation (2).

Next, assume that $\mathrm{den}(x) > \mathrm{den}(y)$.
By the minimality of $m$, we have
\[
\lfloor jx \rfloor = \lfloor jy \rfloor \quad \text{for all } 1 \le j < \mathrm{den}(y).
\]
Since $\mathrm{den}(y)y = \lfloor \mathrm{den}(y)y \rfloor > \mathrm{den}(y)x$, it holds that 
\[
\lfloor \mathrm{den}(y)x \rfloor < \lfloor \mathrm{den}(y)y \rfloor.
\]
Thus, $x$ and $y$ satisfy relation (3).

\end{proof}

Let $x \in \mathbb{Q} \cap (0, 1)$.
Below, we state several fundamental properties of $\{G(x, n)\}_{n\in \mathbb{Z}}$ that will be required later.

\begin{lem}\label{l2}
\begin{enumerate}
\item[\rm (1)] $\{G(x,n)\}_{n\in\mathbb{Z}}$ is a purely periodic sequence with period $\mathrm{den}(x)$.
\item[\rm (2)] $G(x,1)=0$ and $G(x,0)=1$.
\item[\rm (3)] For $n,m\in\mathbb{Z}$, we have
\[
G(x,n+j)=G(x,m+j)\quad\text{for all }j\in\mathbb{Z}_{\geq 0}
\]
if and only if
\[
n\equiv m\pmod{\mathrm{den}(x)}.
\]
\item[\rm (4)] If $m\not\equiv 1\pmod{\mathrm{den}(x)}$, then there exists
$j\in\mathbb{Z}_{\geq 0}$ such that
\[
G(x,1+j)<G(x,m+j),
\]
while
\[
G(x,1+k)=G(x,m+k)
\]
for all $0\leq k<j$.
\item[\rm (5)] If $m\not\equiv 0\pmod{\mathrm{den}(x)}$, then there exists
$j\in\mathbb{Z}_{\geq 0}$ such that
\[
G(x,-j)>G(x,m-j),
\]
while
\[
G(x,-k)=G(x,m-k)
\]
for all $0\leq k<j$.
\item[\rm (6)]
Let $m \in \mathbb{Z}$. Then exactly one of the following two conditions holds:
\begin{enumerate}
\item[\rm (i)] $G(x, -j) = G(x, m+j)$ for all $j \in \mathbb{Z}_{\ge 0}$;
\item[\rm (ii)] there exists $j \in \mathbb{Z}_{\ge 0}$ such that
\[
G(x, -j) > G(x, m+j),
\]
while
\[
G(x, -k) = G(x, m+k)
\]
for all $0 \le k < j$.
\end{enumerate}
\item[\rm (7)] 
Let $x \in \mathbb{Q} \cap (0, 1)$ and $j \in \mathbb{Z}$. Then
\[
G(x, j) = G(x, -j+1) \quad \text{if and only if} \quad j \not\equiv 0,1 \pmod{\mathrm{den}(x)}.
\]
\item[\rm (8)] 
Let $x \in \mathbb{Q} \cap (0, 1)$. For any integer $n$ satisfying $0 < n < \operatorname{den}(x)$, there exists an integer $j > n$ such that
\[
G(x, j) > G(x, -j+1),
\]
while
\[
G(x, k) = G(x, -k+1)
\]
for all integers $k$ with $n < k < j$.

\end{enumerate}
\end{lem}

\begin{proof}
Since statements (1), (2) and (3) are straightforward, we omit their proofs.
We prove (4). Suppose that $m \not\equiv 1 \pmod{\mathrm{den}(x)}$.
By statement (3), there exists a non-negative integer $j \in \mathbb{Z}_{\ge 0}$ such that
\[
G(x, 1+j) \ne G(x, m+j).
\]
Let $j$ be the minimal such non-negative integer. Then, for all $0 \le k < j$, we have
\[
G(x, 1+k) = G(x, m+k).
\]
In this case, we have
\begin{align*}
\sum_{k=0}^j G(x, 1+k) &= \lfloor (1+j)x \rfloor, \\
\sum_{k=0}^j G(x, m+k) &= \lfloor (m+j)x \rfloor - \lfloor (m-1)x \rfloor \\
&= \lfloor (1+j)x + (m-1)x \rfloor - \lfloor (m-1)x \rfloor \\
&= \left\lfloor (1+j)x + \lfloor (m-1)x \rfloor + \{(m-1)x\} \right\rfloor - \lfloor (m-1)x \rfloor \\
&= \left\lfloor (1+j)x + \{(m-1)x\} \right\rfloor \\
&\ge \lfloor (1+j)x \rfloor,
\end{align*}
where $\{\alpha\}$ denotes the fractional part of $\alpha$.
Since $G(x, 1+k) = G(x, m+k)$ for all $0 \le k < j$, subtracting these sums yields
\[
G(x, m+j) \ge G(x, 1+j).
\]
Combined with $G(x, 1+j) \ne G(x, m+j)$, we conclude that
\[
G(x, m+j) > G(x, 1+j).
\]

We now prove (5). First, if $G(x, m) = 0$, the statement holds trivially for $j = 0$; thus, we may assume that $G(x, m) = 1$.

By (3), there exists a integer $j > 0$ such that
\[
G(x, -j) \ne G(x, m-j),
\]
and for all $0 \le k < j$, we have
\[
G(x, -k) = G(x, m-k).
\]

Note that
\[
\sum_{k=0}^j G(x, -k) = \sharp \!\left( ((-j-1)x, 0] \cap \mathbb{Z} \right)
\]
and
\begin{align*}
\sum_{k=0}^j G(x, m-k) 
&= \sum_{k=0}^j \left( \lfloor (m-k)x \rfloor - \lfloor (m-k-1)x \rfloor \right) \\
&= \sum_{k=0}^j \left( \lfloor (m-k)x - \lfloor mx \rfloor \rfloor - \lfloor (m-k-1)x - \lfloor mx \rfloor \rfloor \right) \\
&= \sharp \!\left( \left( (m-j-1)x - \lfloor mx \rfloor, \, mx - \lfloor mx \rfloor \right] \cap \mathbb{Z} \right).
\end{align*}
Since $0 < mx - \lfloor mx \rfloor < 1$, the above set has the same number of integer points as the interval with the upper bound replaced by $0$; that is,
\[
\sum_{k=0}^j G(x, m-k) = \sharp \!\left( \left( (m-j-1)x - \lfloor mx \rfloor, \, 0 \right] \cap \mathbb{Z} \right).
\]
Furthermore, since $(m-j-1)x - \lfloor mx \rfloor = (-j-1)x + (mx - \lfloor mx \rfloor) > (-j-1)x$, we obtain
\[
\sum_{k=0}^j G(x, -k) \ge \sum_{k=0}^j G(x, m-k).
\]
Consequently, we conclude that
\[
G(x, -j) > G(x, m-j).
\]
Since statement (6) can be proved in a similar manner to (5), we omit the proof.

We now prove (7).
By definition, we have
\[
G(x,j) = \lfloor jx \rfloor - \lfloor (j-1)x \rfloor = \lceil (-j+1)x \rceil - \lceil -jx \rceil.
\]

If $j \not\equiv 0, 1 \pmod{\operatorname{den}(x)}$, then neither $(-j+1)x$ nor $-jx$ is an integer. Thus, the ceiling functions coincide with the floor functions, yielding
\[
\lceil (-j+1)x \rceil - \lceil -jx \rceil = \lfloor (-j+1)x \rfloor - \lfloor -jx \rfloor = G(x, -j+1).
\]

Next, suppose $j \equiv 0 \pmod{\operatorname{den}(x)}$. In this case, $jx \in \mathbb{Z}$, which implies
\[
G(x,j) = \lfloor jx \rfloor - \lfloor (j-1)x \rfloor = -\lfloor -x \rfloor = 1.
\]
On the other hand,
\[
G(x,-j+1) = \lfloor (-j+1)x \rfloor - \lfloor -jx \rfloor = \lfloor x \rfloor = 0,
\]
so $G(x,j) \ne G(x,-j+1)$.

Finally, suppose $j \equiv 1 \pmod{\operatorname{den}(x)}$. Then $(j-1)x \in \mathbb{Z}$, and similar computations yield
\[
G(x,j) = \lfloor jx \rfloor - \lfloor (j-1)x \rfloor = \lfloor x \rfloor = 0.
\]
On the other hand, we have
\[
G(x,-j+1) = \lfloor (-j+1)x \rfloor - \lfloor -jx \rfloor = -\lfloor -x \rfloor = 1,
\]
which again shows $G(x,j) \ne G(x,-j+1)$. This completes the proof of (7).

We now prove (8).
It suffices to set $j = \mathrm{den}(x)$. From the proof of (7), we directly have
\[
G(x, \mathrm{den}(x)) = 1 > 0 = G(x, -\mathrm{den}(x)+1).
\]
Furthermore, by (7), for any integer $k$ with $n < k < j = \mathrm{den}(x)$, since $k \not\equiv 0, 1 \pmod{\mathrm{den}(x)}$, we obtain
\[
G(x, k) = G(x, -k+1),
\]
which completes the proof.

\end{proof}

Although it is well known that the sequences $\{G(x,n)\}_{n \in \mathbb{Z}}$ and $\{G(x,-n)\}_{n \in \mathbb{Z}}$ are essentially the same, we provide a proof for completeness.

\begin{lem}\label{suitableindex}
Let $x \in \mathbb{Q} \cap [0, 1]$.
The sequence $\{G(x,-n)\}_{n \in \mathbb{Z}}$ is obtained by shifting the sequence $\{G(x,n)\}_{n \in \mathbb{Z}}$ by a suitable index.
\end{lem}
\begin{proof}
If $x = 0$ or $x = 1$, the claim of the lemma holds trivially.
Suppose that $0 < x < 1$.
For any $n \in \mathbb{Z}$, since
\[
\lfloor nx \rfloor - \lfloor (n-1)x \rfloor = \lceil (-n+1)x \rceil - \lceil -nx \rceil,
\]
it follows that
\[
\{G(x,-n)\}_{n \in \mathbb{Z}} = \{\lceil (n+1)x \rceil - \lceil nx \rceil\}_{n \in \mathbb{Z}}.
\]
Therefore, it suffices to show that shifting the sequence $\{\lceil nx \rceil - \lceil (n-1)x \rceil\}_{n \in \mathbb{Z}}$ by a suitable index yields $\{G(x,n)\}_{n \in \mathbb{Z}}$.
Let $x = p/q$, where $p$ and $q$ are coprime positive integers.
Let $k$ be an integer satisfying $kp \equiv 1 \pmod{q}$.
Set $l = (kp - 1)/q \in \mathbb{Z}$.
Then we have
\begin{align*}
\lceil (n+k)x \rceil &= \left\lceil nx + \frac{kp}{q} \right\rceil \\
&= \left\lceil nx + \frac{ql+1}{q} \right\rceil \\
&= \left\lceil nx + \frac{1}{q} \right\rceil + l.
\end{align*}
Therefore, it follows that
\[
\lceil (n+k)x \rceil - \lceil (n-1+k)x \rceil = \left\lceil nx + \frac{1}{q} \right\rceil - \left\lceil (n-1)x + \frac{1}{q} \right\rceil.
\]
Next, we show that
\begin{align}\label{leftlceilnx}
\left\lceil nx + \frac{1}{q} \right\rceil - \left\lceil (n-1)x + \frac{1}{q} \right\rceil = \lfloor nx \rfloor - \lfloor (n-1)x \rfloor.
\end{align}
Suppose that $\lfloor nx \rfloor = \lfloor (n-1)x \rfloor$.
This is equivalent to
\[
nx \in [\lfloor (n-1)x \rfloor, \lfloor (n-1)x \rfloor + 1),
\]
which is further equivalent to
\[
nx \in \left\{ \lfloor (n-1)x \rfloor + \frac{k}{q} \;\middle|\; 0 \le k \le q-1 \right\}.
\]
This holds if and only if
\[
\left\lceil nx + \frac{1}{q} \right\rceil = \left\lceil (n-1)x + \frac{1}{q} \right\rceil.
\]
Therefore, \eqref{leftlceilnx} holds.
This completes the proof of the lemma.

\end{proof}

\begin{lem}\label{exactlyone}
Let $x \in \mathbb{Q} \cap [0, 1]$.
Let $m \in \mathbb{Z}$.
Then, exactly one of the following two conditions holds:
\begin{enumerate}
    \item $G(x, k) = G(x, m - k)$ for all $k\in \mathbb{Z}$.
    \item There exists $j \in \mathbb{Z}_{> 0}$ such that
    \[
    G(x, j) < G(x, m - j)
    \]
    and
    \[
    G(x, k) = G(x, m - k) \quad \text{for all } 0 < k < j.
    \]
\end{enumerate}
\end{lem}
\begin{proof}
By Lemma \ref{suitableindex}, the sequence $\{G(x,-n)\}_{n \in \mathbb{Z}}$ is a shift of the sequence $\{G(x,n)\}_{n \in \mathbb{Z}}$.
Hence, the claim follows from Lemma \ref{l2} (3) and (4). 
\end{proof}

\section{Lagrange constants} \label{sec:lagrange}

We consider the following regular continued fraction expansion:
\begin{equation*}
[0;a_1,a_2,a_3,\ldots] = \cfrac{1}{a_1 + \cfrac{1}{a_2 + \cfrac{1}{a_3+\cfrac{1}{\ddots}}}},
\end{equation*}
where $a_i \in \mathbb{Z}_{>0}$ for $i = 1, 2, \ldots$.
Let $p_n/q_n = [0; a_1, \ldots, a_n]$ denote the $n$-th convergent.
Then, as is well known, for $n \in \mathbb{Z}_{>0}$, $p_n$ and $q_n$ are defined by
\begin{equation*}
\begin{pmatrix}
p_{n-1} & p_n \\
q_{n-1} & q_n
\end{pmatrix}
=
\begin{pmatrix}
0 & 1 \\
1 & a_1
\end{pmatrix}
\begin{pmatrix}
0 & 1 \\
1 & a_2
\end{pmatrix}
\cdots
\begin{pmatrix}
0 & 1 \\
1 & a_n
\end{pmatrix},
\end{equation*}
where $p_0 = 0$ and $q_0 = 1$.

When a sequence of positive integers $a_1, a_2, \ldots$ is purely periodic, it naturally extends to a doubly infinite sequence $\{a_i\}_{i \in \mathbb{Z}}$. 
For such a sequence, the following fact is well known.

\begin{lem}\label{l3}
For $\alpha = [0; a_1, a_2, \ldots]$, we have
\begin{equation}
\mathcal{L}(\alpha) = \max_{i \in \mathbb{Z}} \left( [0; a_{i-1}, a_{i-2}, \ldots] + a_i + [0; a_{i+1}, a_{i+2}, \ldots] \right).
\end{equation}
\end{lem}

Let $x \in \mathbb{Q} \cap [0; 1]$, and let $a, b \in \mathbb{Z}_{>0}$ satisfy $a < b$.
Set $[\phi_{a,b}(G(x))] = [0; a_1, a_2, \ldots]$.
Since the sequence $a_1, a_2, \ldots$ is purely periodic, we naturally extend it to a doubly infinite sequence $\{a_i\}_{i \in \mathbb{Z}}$.
That is, for each $n \in \mathbb{Z}$, the pair $a_{2n-1} a_{2n}$ is given by
\[
a_{2n-1} a_{2n} = 
\begin{cases}
aa & \text{if } G(x, n) = 0, \\
bb & \text{if } G(x, n) = 1.
\end{cases}
\]
For $n \in \mathbb{Z}$, we define
\[
L(x, n) := [0; a_{n-2}, a_{n-3}, \ldots] + a_{n-1} + [0; a_n, a_{n+1}, \ldots].
\]

\begin{lem}\label{l4}
For $n, m \in \mathbb{Z}$, if $a_{n-1} < a_{m-1}$, then
\[
L(x, n) < L(x, m).
\]
\end{lem}

\begin{proof}
Suppose that $a_{n-1} + 2 \le a_{m-1}$.
Then we have
\[
L(x, n) = [0; a_{n-2}, a_{n-3}, \ldots] + a_{n-1} + [0; a_n, a_{n+1}, \ldots]
< 2 + a_{n-1} \le a_{m-1} < L(x, m).
\]
Next, suppose that $a_{n-1} + 1 = a_{m-1}$.
In this case, we have $b = a + 1$, with $a_{n-1} = a$ and $a_{m-1} = b$.
Then, the following bounds hold:
\begin{align}\nonumber
L(x, n) &= [0; a_{n-2}, a_{n-3}, \ldots] + a + [0; a_n, a_{n+1}, \ldots] \\
&\le [0; \overline{a, b}] + a + [0; \overline{a, b}], \label{L(x,n)}\\
\nonumber
L(x, m) &= [0; a_{m-2}, a_{m-3}, \ldots] + a + 1 + [0; a_m, a_{m+1}, \ldots] \\
&\ge [0; \overline{b, a}] + a + 1 + [0; \overline{b, a}].\label{L(x,m)}
\end{align}
Now, assume that $a \ge 2$.
Since $[0; \overline{a, b}] < 1/2$, it follows from \eqref{L(x,n)} that
\[
L(x, n) < a + 1,
\]
which yields $L(x, n) < L(x, m)$.
Next, assume that $a = 1$.
Since
\[
[0; \overline{1, 2}] = \sqrt{3} - 1 \quad \text{and} \quad [0; \overline{2, 1}] = \frac{\sqrt{3} - 1}{2},
\]
we have
\begin{align*}
L(x, n) &\le 2(\sqrt{3} - 1) + 1 = 2\sqrt{3} - 1, \\
L(x, m) &\ge 2 \cdot \frac{\sqrt{3} - 1}{2} + 2 = \sqrt{3} + 1.
\end{align*}
Since $2\sqrt{3} - 1 < \sqrt{3} + 1$, this  implies that
\[
L(x, n) < L(x, m).
\]

\end{proof}

\begin{lem}\label{l5}
For $n\in \mathbb{Z}$,
\[
L(x, 1) \geq L(x, 2n+1).
\]
\end{lem}
\begin{proof}
If $x = 0$ or $x = 1$, the claim of the lemma holds trivially.
We assume that $0<x<1$.
By the periodicity of $\{a_i\}_{i \in \mathbb{Z}}$, we may assume without loss of generality that $n > 0$.
By (3) of Lemma \ref{l2}, if $n \equiv 0 \pmod{\mathrm{den}(x)}$, then we have
\[
L(x, 1) = L(x, 2n+1).
\]
Now, suppose that $n \not\equiv 0 \pmod{\mathrm{den}(x)}$.
By (4) of Lemma \ref{l2}, there exists
$j\in\mathbb{Z}_{\geq 0}$ such that
\[
G(x,1+j)<G(x,n+1+j),
\]
while
\[
G(x,1+k)=G(x,n+1+k)
\]
for all $0\leq k<j$.
Therefore,  we have
\[
a_1 a_2 \ldots a_{2j} = a_{2n+1} a_{2n+2} \ldots a_{2n+2j} \quad \text{and} \quad a_{2j+1} = a, \quad a_{2n+2j+1} = b.
\]
When $j = 0$, we simply have $a_1 = a$ and $a_{2n+1} = b$.

Since $a_{2j+1} < a_{2n+2j+1}$, it follows that
\begin{align}\label{a_1a_2}
[0; a_1, a_2, \ldots, a_{2j}, a_{2j+1}, \ldots] > [0; a_{2n+1}, a_{2n+2}, \ldots, a_{2n+2j}, a_{2n+2j+1}, \ldots].
\end{align}
If $G(x, n) = 0$, then we have
\[
a_0 = b, \quad a_{2n} = a,
\]
and by Lemma \ref{l4}, it follows that
\[
L(x, 1) > L(x, 2n+1).
\]
We assume that $G(x, n) = 1$.
By (5) of Lemma \ref{l2}, then there exists
$j\in\mathbb{Z}_{> 0}$ such that
\[
G(x,-j)>G(x,n-j),
\]
while
\[
G(x,-k)=G(x,n-k)
\]
for all $0\leq k<j$.
Therefore,  we have
\[
a_{-1} a_{-2} \ldots a_{-2j+1} = a_{2n-1} a_{2n-2} \ldots a_{2n-2j+1} \quad \text{and} \quad a_{-2j} = b, \quad a_{2n-2j} = a.
\]
Since $a_{-2j} > a_{2n-2j}$, it follows that
\begin{align}\label{a_-1a_-2}
[0;a_{-1}, a_{-2}, \ldots, a_{-2j+1}, a_{-2j}, \ldots] > [0; a_{2n-1}, a_{2n-2}, \ldots, a_{2n-2j+1}, a_{2n-2j}, \ldots].
\end{align}
Therefore, combining \eqref{a_1a_2} and \eqref{a_-1a_-2} with $a_0 = a_{2n} = b$, we obtain
\[
L(x, 1) > L(x, 2n+1).
\]
\end{proof}
\begin{lem}\label{l6}
For $n\in \mathbb{Z}$,
\[
L(x, 1) \geq L(x, 2n).
\]
\end{lem}
\begin{proof}
If $x = 0$ or $x = 1$, the claim of the lemma holds trivially.
We assume that $0<x<1$.
By the periodicity of $\{a_i\}_{i \in \mathbb{Z}}$, we may assume without loss of generality that $n > 0$.
If $G(x, n) = 0$, then we have
\[
a_0 = b, \quad a_{2n-1} = a,
\]
and by Lemma \ref{l4}, it follows that
\[
L(x, 1) > L(x, 2n).
\]
We assume that $G(x, n) = 1$.

First, suppose that the first case of Lemma~\ref{l2}(6) holds with $m = n$, that is,
\[
G(x, -j) = G(x, n+j) \quad \text{for all } j \in \mathbb{Z}_{\ge 0}.
\]
Then we have the equality of sequences
\[
a_0 a_{-1} a_{-2} \dots = a_{2n-1} a_{2n} a_{2n+1} \dots,
\]
which implies that
\begin{equation}\label{a_0+00}
a_0 + [0; a_{-1}, a_{-2}, \dots] = a_{2n-1} + [0; a_{2n}, a_{2n+1}, \dots].
\end{equation}

Next, suppose that the second case of Lemma~\ref{l2}(6) holds with $m = n$. Namely, there exists $j \in \mathbb{Z}_{\ge 0}$ such that
\[
G(x, -j) > G(x, n+j)
\]
and
\[
G(x, -k) = G(x, n+k) \quad \text{for all } 0 \le k < j.
\]
Note that $j \ge 1$ since $G(x, n) = 1$. Therefore, we have
\[
a_0 a_{-1} \dots a_{-2j+1} = a_{2n-1} a_{2n} \dots a_{2n+2j-2} \quad \text{and} \quad a_{-2j} = b, \quad a_{2n+2j-1} = a.
\]
This implies that
\begin{equation}\label{a_0+01}
a_0 + [0; a_{-1}, a_{-2}, \dots] > a_{2n-1} + [0; a_{2n}, a_{2n+1}, \dots].
\end{equation}

In either case, we obtain
\begin{equation}\label{a_0+0}
a_0 + [0; a_{-1}, a_{-2}, \dots] \ge a_{2n-1} + [0; a_{2n}, a_{2n+1}, \dots].
\end{equation}
Suppose that condition (1) of Lemma \ref{exactlyone} holds for $m = n$, that is,
\[
G(x, k) = G(x, n - k) \quad \text{for all } k \in \mathbb{Z}.
\]
Therefore, we have
\[
a_1 a_2 \dots = a_{2n-2} a_{2n-3} \dots,
\]
which implies that
\begin{align}\label{0a_1}
[0;a_1, a_2, \dots] = [0;a_{2n-2}, a_{2n-3}, \dots].
\end{align}
From \eqref{a_0+0} and \eqref{0a_1}, we obtain
\[
L(x, 1) \geq L(x, 2n).
\]
Suppose that condition (2) of Lemma \ref{exactlyone} holds for $m = n$, that is,
There exists $j' \in \mathbb{Z}_{> 0}$ such that
    \[
    G(x, j') < G(x, n - j')
    \]
    and
    \[
    G(x, k) = G(x, n - k) \quad \text{for all } 0 < k < j'.
    \]
Therefore, we have $a_{2j'-1} = a$ and $a_{2n-2j'} = b$.
Furthermore, when $j' \ge 2$, it holds that
\[
a_1 a_2 \dots a_{2j'-2} = a_{2n-2} a_{2n-3} \dots a_{2n-(2j'-1)}.
\]
This implies that
\begin{align}\label{0a_1>}
[0; a_1, a_2, \dots] > [0; a_{2n-2}, a_{2n-3}, \dots].
\end{align}
From \eqref{a_0+0} and \eqref{0a_1>}, we obtain
\[
L(x, 1) > L(x, 2n).
\]

\end{proof}

Combining Lemmas \ref{l5} and \ref{l6}, we immediately obtain the following theorem.

\begin{thm}\label{t3}
Let $a,b \in \mathbb{Z}_{>0}$ with $a<b$.
Let $x \in \mathbb{Q}$ such that $0 \le x \le 1$.
Then,
\[
\mathcal{L}([\phi_{a,b}(G(x))]) = L(x,1).
\]
\end{thm}

\section{Proof of Main Theorem}\label{sec:main}

In this section, we prove the main theorem by dividing it into three cases, each presented as a lemma.
We define
\[
\overline{G}(x) = \dots G(x,-1) G(x,0) G(x,1) \dots.
\]

\begin{lem}\label{ml1}
Let $x, y \in \mathbb{Q}$ satisfy $0 \le x < y \le 1$.
If $x$ and $y$ are of type (1) and $a \in \mathbb{Z}_{>0}$ satisfies $a \ge 2$, then
\[
\mathcal{L}([\phi_{a,a+1}(G(x))]) < \mathcal{L}([\phi_{a,a+1}(G(y))]).
\]
\end{lem}
\begin{proof}
Let
\begin{align}\label{aa+1}
\phi_{a,a+1}(\overline{G}(x)) = \dots a_{-1} a_0 a_1 a_2 \dots
\end{align}
and
\begin{align}\label{ba+1}
\phi_{a,a+1}(\overline{G}(y)) = \dots b_{-1} b_0 b_1 b_2 \dots.
\end{align}
Here, we set
\[
\phi_{a,a+1}(G(x,k)) = a_{2k-1} a_{2k} \quad \text{and} \quad \phi_{a,a+1}(G(y,k)) = b_{2k-1} b_{2k}.
\]
There exists $n < \min(\mathrm{den}(x), \mathrm{den}(y))$ satisfying $\lfloor nx \rfloor < \lfloor ny \rfloor$.
In particular, let $n$ be the smallest such integer.
We note that $n\geq 2$ and $0<x,y<1$.
Then, for all $0 \le k < n$, we have
\[
\lfloor kx \rfloor = \lfloor ky \rfloor.
\]
Then, one can easily verify that $\lfloor -kx \rfloor = \lfloor -ky \rfloor$ for all $0 \le k < n$, as well as $\lfloor -nx \rfloor > \lfloor -ny \rfloor$.
From the above considerations and (7) of Lemma \ref{l2}, we obtain the following properties:
\begin{enumerate}
    \item $G(x,0) = 1$, $G(x,1) = 0$, $G(y,0) = 1$, and $G(y,1) = 0$;
    \item $G(x,n) = 0$ and $G(y,n) = 1$;
    \item $G(x,2) \dots G(x,n) = G(x,-1) \dots G(x,-(n-1))$;
    \item $G(y,2) \dots G(y,n) = G(y,-1) \dots G(y,-(n-1))$;
    \item $G(x,2) \dots G(x,n-1) = G(y,2) \dots G(y,n-1)$.
\end{enumerate}
Thus, we have
\[
\begin{aligned}
a_1 a_2 a_3 \dots a_{2n} &= a a a_3 \dots a_{2n-2} a a, \\
b_1 b_2 b_3 \dots b_{2n} &= a a a_3 \dots a_{2n-2} (a+1)(a+1),
\end{aligned}
\]
and
\[
\begin{aligned}
a_{-1} a_{-2} \dots a_{-2n+1} &= (a+1) a_3 \dots a_{2n-2} a a, \\
b_{-1} b_{-2} \dots b_{-2n+1} &= (a+1) a_3 \dots a_{2n-2} (a+1)(a+1).
\end{aligned}
\]
Here and in what follows, a word such as $a_3 \dots a_{2n-2}$ is understood to be the empty word when $n = 2$.

Therefore, we obtain
\begin{align}\label{08251}
\begin{aligned}
[0; a_1, a_2, a_3, \dots, a_{2n}, x_1]
&= [0; a, a, a_3, \dots, a_{2n-2}, a, a, x_1], \\[1ex]
[0; b_1, b_2, b_3, \dots, b_{2n}, y_1]
&= [0; a, a, a_3, \dots, a_{2n-2}, a+1, a+1, y_1], \\[1ex]
[0; a_{-1}, a_{-2}, \dots, a_{-2n+1}, x_2]
&= [0; a+1, a_3, \dots, a_{2n-2}, a, a, x_2], \\[1ex]
[0; b_{-1}, b_{-2}, \dots, b_{-2n+1}, y_2]
&= [0; a+1, a_3, \dots, a_{2n-2}, a+1, a+1, y_2],
\end{aligned}
\end{align}
where
\[
\begin{aligned}
x_1 &= [a_{2n+1}; a_{2n+2}, \dots], & y_1 &= [b_{2n+1}; b_{2n+2}, \dots], \\
x_2 &= [a_{-2n}; a_{-2n-1}, \dots], & y_2 &= [b_{-2n}; b_{-2n-1}, \dots].
\end{aligned}
\]
Let
\begin{align}\label{m_1m_219}
\begin{pmatrix} m_1 & m_2 \\ m_3 & m_4 \end{pmatrix}
= \begin{pmatrix} 0 & 1 \\ 1 & a_3 \end{pmatrix} \dots \begin{pmatrix} 0 & 1 \\ 1 & a_{2n-2} \end{pmatrix},
\end{align}
where this matrix product is understood to be the identity matrix when $n = 2$.

Therefore, by \eqref{08251}, we have
\[
\begin{aligned}
&\begin{pmatrix} 0 & 1 \\ 1 & a_1 \end{pmatrix} \dots \begin{pmatrix} 0 & 1 \\ 1 & a_{2n} \end{pmatrix} \begin{pmatrix} 1 \\ x_1 \end{pmatrix} \\
&\quad = \begin{pmatrix} 0 & 1 \\ 1 & a \end{pmatrix}^2 \begin{pmatrix} m_1 & m_2 \\ m_3 & m_4 \end{pmatrix} \begin{pmatrix} 0 & 1 \\ 1 & a \end{pmatrix}^2 \begin{pmatrix} 1 \\ x_1 \end{pmatrix} \\
&\quad = \begin{pmatrix}
P_1\\
Q_1
\end{pmatrix},
\end{aligned}
\]
where
\[
\begin{aligned}\label{a^2m_4}
P_1 &= a^2 m_4 + a(m_2 + m_3) + m_1 + \bigl( a^3 m_4 + a^2(m_2 + m_3) + a(m_1 + m_4) + m_2 \bigr) x_1, \\[1ex]
Q_1 &= a^3 m_4 + a^2(m_2 + m_3) + a(m_1 + m_4) + m_3 \\
  &\quad + \bigl( (a^4 + 2a^2 + 1)m_4 + (a^3 + a)(m_2 + m_3) + a^2 m_1 \bigr) x_1.
\end{aligned}
\]
Therefore, we obtain
\[
[0; a_1, a_2, a_3, \dots, a_{2n}, x_1] = \frac{P_1}{Q_1}.
\]

By \eqref{08251}, we have
\[
\begin{aligned}
\begin{pmatrix} 0 & 1 \\ 1 & b_1 \end{pmatrix} \dots \begin{pmatrix} 0 & 1 \\ 1 & b_{2n} \end{pmatrix} \begin{pmatrix} 1 \\ y_1 \end{pmatrix}
&= \begin{pmatrix} 0 & 1 \\ 1 & a \end{pmatrix}^2 \begin{pmatrix} m_1 & m_2 \\ m_3 & m_4 \end{pmatrix} \begin{pmatrix} 0 & 1 \\ 1 & a+1 \end{pmatrix}^2 \begin{pmatrix} 1 \\ y_1 \end{pmatrix} \\[1ex]
&= \begin{pmatrix} P_2 \\ Q_2 \end{pmatrix},
\end{aligned}
\]
where
\[
\begin{aligned}
P_2 &= a^2 m_4 + a(m_2+m_3+m_4) + m_1+m_2 \\
   &\quad + \bigl( (a^3+2a^2)m_4 + a^2(m_2+m_3) + a(m_1+2m_2+m_3+2m_4) + m_1+2m_2 \bigr) y_1, \\[1.5ex]
Q_2 &= a^3 m_4 + a^2(m_2+m_3+m_4) + a(m_1+m_2+m_4) + m_3+m_4 \\
   &\quad + \bigl( (a^4+2a^3+3a^2+2a)m_4 + (a^3+2a^2+a)m_2 + (a^3+a^2+a)m_3 + (a^2+a)m_1 + m_3+2m_4 \bigr) y_1.
\end{aligned}
\]

Therefore, we obtain
\[
[0; b_1, b_2, b_3, \dots, b_{2n}, y_1] = \frac{P_2}{Q_2}.
\]

By \eqref{08251}, we have
\[
\begin{aligned}
&\begin{pmatrix} 0 & 1 \\ 1 & a_{-1} \end{pmatrix} \dots \begin{pmatrix} 0 & 1 \\ 1 & a_{-2n+1} \end{pmatrix} \begin{pmatrix} 1 \\ x_2 \end{pmatrix} \\
&\quad = \begin{pmatrix} 0 & 1 \\ 1 & a+1 \end{pmatrix} \begin{pmatrix} m_1 & m_2 \\ m_3 & m_4 \end{pmatrix} \begin{pmatrix} 0 & 1 \\ 1 & a \end{pmatrix}^2 \begin{pmatrix} 1 \\ x_2 \end{pmatrix} \\
&\quad = \begin{pmatrix}
P_1'\\
Q_1'
\end{pmatrix},
\end{aligned}
\]
where
\[
\begin{aligned}\label{a^2m_4}
P_1' &= a m_4 + m_3 + \bigl( (a^2 + 1) m_4 + a m_3 \bigr) x_2, \\[1ex]
Q_1' &= a^2 m_4 + a(m_2 + m_3 + m_4) + m_1 + m_3 \\
    &\quad + \bigl( a^3 m_4 + a^2(m_2 + m_3 + m_4) + a(m_1 + m_3 + m_4) + m_2 + m_4 \bigr) x_2.
\end{aligned}
\]

Therefore, we obtain
\[
[0; a_{-1}, a_{-2}, \dots, a_{-2n+1}, x_2] = \frac{P_1'}{Q_1'}.
\]

By \eqref{08251}, we have
\[
\begin{aligned}
\begin{pmatrix} 0 & 1 \\ 1 & b_{-1} \end{pmatrix} \dots \begin{pmatrix} 0 & 1 \\ 1 & b_{-2n+1} \end{pmatrix} \begin{pmatrix} 1 \\ y_2 \end{pmatrix}
&= \begin{pmatrix} 0 & 1 \\ 1 & a+1 \end{pmatrix} \begin{pmatrix} m_1 & m_2 \\ m_3 & m_4 \end{pmatrix} \begin{pmatrix} 0 & 1 \\ 1 & a+1 \end{pmatrix}^2 \begin{pmatrix} 1 \\ y_2 \end{pmatrix} \\[1ex]
&= \begin{pmatrix} P_2' \\ Q_2' \end{pmatrix},
\end{aligned}
\]
where
\[
\begin{aligned}
P_2' &= a m_4 + m_3 + m_4 
      + \bigl( a^2 m_4 + a(m_3 + 2m_4) + m_3 + 2m_4 \bigr) y_2, \\[1.5ex]
Q_2' &= a^2 m_4 + a(m_2 + m_3 + 2m_4) + m_1 + m_2 + m_3 + m_4 \\
     &\quad + \Bigl( a^3 m_4 + a^2(m_2 + m_3 + 3m_4) + a(m_1 + 2m_2 + 2m_3 + 4m_4) \\
     &\qquad\quad + m_1 + 2m_2 + m_3 + 2m_4 \Bigr) y_2.
\end{aligned}
\]
Therefore, we obtain
\[
[0; b_{-1}, b_{-2}, \dots, b_{-2n+1}, y_2] = \frac{P_2'}{Q_2'}.
\]

Therefore, considering $m_1 m_4-m_2 m_3=1$, we obtain
\begin{align}\label{P_2Q_219}
\frac{P_2}{Q_2} - \frac{P_1}{Q_1}
= -\frac{(a^2 + a - 1) x_1 y_1 + (a-1)x_1 + (a+2)y_1 + 1}{Q_1 Q_2}.
\end{align}
Let $R$ denote the numerator $(a^2 + a - 1) x_1 y_1 + (a-1)x_1 + (a+2)y_1 + 1$.

Similarly, we have
\begin{align}\label{P'_2Q'_219}
\frac{P_2'}{Q_2'} - \frac{P_1'}{Q_1'}
= \frac{(a^2 + a - 1) x_2 y_2 + (a-1)x_2 + (a+2)y_2 + 1}{Q_1' Q_2'}.
\end{align}
Let $R'$ denote the numerator $(a^2 + a - 1) x_2 y_2 + (a-1)x_2 + (a+2)y_2 + 1$.

It is clear that
\[
\frac{P_2}{Q_2} - \frac{P_1}{Q_1} < 0 \quad \text{and} \quad \frac{P_2'}{Q_2'} - \frac{P_1'}{Q_1'} > 0.
\]
By Theorem~\ref{t3}, the Lagrange constants are given by
\[
\begin{aligned}
\mathcal{L}([\phi_{a,b}(G(x))]) &= a_0 + [0; a_1, a_2, a_3, \dots, a_{2n}, x_1] + [0; a_{-1}, a_{-2}, \dots, a_{-2n+1}, x_2], \\
\mathcal{L}([\phi_{a,b}(G(y))]) &= b_0 + [0; b_1, b_2, b_3, \dots, b_{2n}, y_1] + [0; b_{-1}, b_{-2}, \dots, b_{-2n+1}, y_2].
\end{aligned}
\]
Since $a_0 = b_0$, to establish the inequality
\[
\mathcal{L}([\phi_{a,b}(G(x))]) < \mathcal{L}([\phi_{a,b}(G(y))]),
\]
it suffices to show that
\begin{align}\label{fracP_2}
\frac{P_2'}{Q_2'} - \frac{P_1'}{Q_1'} > -\left( \frac{P_2}{Q_2} - \frac{P_1}{Q_1} \right),
\end{align}
that is,
\begin{equation}\label{0919-1}
\frac{R'}{Q_1'Q_2'} > \frac{R}{Q_1Q_2}.
\end{equation}
To evaluate this inequality, we first compare $x_1$ and $x_2$.

We apply Lemma~\ref{l2}(8).
Hence, there exists an integer $j > n$ such that
\[
G(x, j) > G(x, -j+1)
\]
and
\[
G(x, k) = G(x, -k+1) \quad \text{for all } k \text{ satisfying } n < k < j.
\]
We have
\[
a_{2n+1} a_{2n+2} a_{2n+3} \dots = \phi_{a,a+1}(G(x, n+1) G(x, n+2) \dots)
\]
and
\[
a_{-2n} a_{-2n-1} a_{-2n-2} \dots = \phi_{a,a+1}(G(x, -n) G(x, -n-1) \dots).
\]
Since
\[
x_1 = [a_{2n+1}; a_{2n+2}, a_{2n+3}, \dots]
\quad \text{and} \quad
x_2 = [a_{-2n}; a_{-2n-1}, a_{-2n-2}, \dots],
\]
we obtain
\begin{equation}\label{0919-2}
x_1 > x_2.
\end{equation}
Similarly, we have
\begin{equation}\label{0919-3}
y_1 > y_2.
\end{equation}

Since the partial quotients of the continued fraction expansions of $x_1$ and $x_2$ consist of $a$ and $a+1$, we have
\[
x_2 > [a; a+2] = a + \frac{1}{a+2}
\]
and
\[
x_1 < [a+1; a] = a + 1 + \frac{1}{a}.
\]
Therefore, we obtain
\begin{align}\label{a+1}
\frac{x_1}{x_2} < \frac{a + 1 + 1/a}{a + 1/(a+2)}.
\end{align}
For $a \ge 1$, it is straightforward to verify that
\begin{align}
\frac{a + 1 + 1/a}{a + 1/(a+2)} < 1 + \frac{2}{a},
\end{align}
which implies
\begin{align}\label{x_2<}
x_1 < \left( 1 + \frac{2}{a} \right) x_2.
\end{align}
Similarly, we have
\begin{align}\label{y_2<}
y_1 < \left( 1 + \frac{2}{a} \right) y_2.
\end{align}
From this point onward, we assume that $a \ge 3$. 
The case $a = 2$ will be treated separately. 
We now compare the numerators and denominators of \eqref{P_2Q_219} and \eqref{P'_2Q'_219}, respectively.

By \eqref{x_2<} and \eqref{y_2<}, we obtain
\begin{align}\label{x_1y_1+}
\begin{aligned}
&\left( 1 + \frac{2}{a} \right)^2R'\\
&=\left( 1 + \frac{2}{a} \right)^2 \bigl( (a^2 + a - 1) x_2 y_2 + (a - 1) x_2 + (a + 2) y_2 + 1 \bigr) \\
&= (a^2 + a - 1) \left( \left( 1 + \frac{2}{a} \right) x_2 \right) \left( \left( 1 + \frac{2}{a} \right) y_2 \right) \\
&\quad + (a - 1) \left( 1 + \frac{2}{a} \right)^2 x_2 + (a + 2) \left( 1 + \frac{2}{a} \right)^2 y_2 + \left( 1 + \frac{2}{a} \right)^2 \\
&> (a^2 + a - 1) x_1 y_1 + (a - 1) x_1 + (a + 2) y_1 + 1\\
&=R.\\
\end{aligned}
\end{align}

Note that, by \eqref{m_1m_219}, we have
\[
m_1 \ge 1, \quad m_2 \ge 0, \quad m_3 \ge 0, \quad \text{and} \quad m_4 \ge 1.
\]
We rewrite $Q_1$ and $Q_1'$ as follows:
\[
\begin{aligned}
Q_1 &= (a^4 + 2a^2 + 1) m_4 x_1 + (a^3 + a) m_3 x_1 + (a^3 + a) m_2 x_1 + a^2 m_1 x_1 \\
    &\quad + (a^3 + a) m_4 + (a^2 + 1) m_3 + a^2 m_2 + a m_1, \\[1.5ex]
Q_1' &= (a^3 + a^2 + a + 1) m_4 x_2 + (a^2 + a) m_3 x_2 + (a^2 + 1) m_2 x_2 + a m_1 x_2 \\
     &\quad + (a^2 + a) m_4 + (a + 1) m_3 + a m_2 + m_1.
\end{aligned}
\]
For $a \ge 3$, it is straightforward to verify that the following inequalities hold:
\[
\begin{aligned}
\left( a^3 + a^2 + a + 1 \right) \left( 1 + \frac{2}{a} \right) &< a^4 + 2a^2 + 1, \\
(a^2 + a) \left( 1 + \frac{2}{a} \right) &< a^3 + a, \\
(a^2 + 1) \left( 1 + \frac{2}{a} \right) &< a^3 + a, \\
a \left( 1 + \frac{2}{a} \right) &< a^2, \\
(a + 1) \left( 1 + \frac{2}{a} \right) &< a^2 + 1.\\
\left( 1 + \frac{2}{a} \right) &<a
\end{aligned}
\]
Consequently, considering $x_1>x_2$, we obtain
\begin{align}\label{left(1}
\left( 1 + \frac{2}{a} \right) Q_1' < Q_1.
\end{align}
Similarly, we rewrite $Q_2$ and $Q_2'$ as follows:
\[
\begin{aligned}
Q_2 &= (a^4 + 2a^3 + 3a^2 + 2a+2) m_4 y_1 + (a^3 + 2a^2 + 2a) m_2 y_1 + (a^3 + a^2 + a+1) m_3 y_1 \\
    &\quad + (a^2 + a) m_1 y_1 + (a^3 + a^2 + a + 1) m_4 + (a^2 + a) m_2 + (a^2 + 1) m_3 + a m_1, \\[1.5ex]
Q_2' &= (a^3 + 3a^2 + 4a + 2) m_4 y_2 + (a^2 + 2a + 2) m_2 y_2 + (a^2 + 2a + 1) m_3 y_2 \\
     &\quad + (a + 1) m_1 y_2 + (a^2 + 2a + 1) m_4 + (a + 1) m_3 + (a + 1) m_2 + m_1.
\end{aligned}
\]
Similarly, we obtain
\begin{equation}\label{left(1+}
\left( 1 + \frac{2}{a} \right) Q_2' < Q_2.
\end{equation}

Using the inequalities \eqref{x_1y_1+}, \eqref{left(1}, and \eqref{left(1+}, we can establish \eqref{0919-1}.
That is, the lemma holds for all $a \ge 3$.

Now consider the case $a = 2$.
In this case, $R$ and $R'$ are given by
\[
\begin{aligned}
R &= 5x_1 y_1 + x_1 + 4y_1 + 1, \\
R' &= 5x_2 y_2 + x_2 + 4y_2 + 1,
\end{aligned}
\]
and the denominators $Q_1, Q_1', Q_2, Q_2'$ reduce to
\[
\begin{aligned}
Q_1 &= 4m_1 x_1 + 10m_2 x_1 + 10m_3 x_1 + 25m_4 x_1 + 2m_1 + 4m_2 + 5m_3 + 10m_4, \\
Q_1' &= 2m_1 x_2 + 5m_2 x_2 + 6m_3 x_2 + 15m_4 x_2 + m_1 + 2m_2 + 3m_3 + 6m_4, \\[1ex]
Q_2 &= 6m_1 y_1 + 20m_2 y_1 + 15m_3 y_1 + 50m_4 y_1 + 2m_1 + 6m_2 + 5m_3 + 15m_4, \\
Q_2' &= 3m_1 y_2 + 10m_2 y_2 + 9m_3 y_2 + 30m_4 y_2 + m_1 + 3m_2 + 3m_3 + 9m_4.
\end{aligned}
\]
Furthermore, by \eqref{a+1}, we have
\begin{equation}
\frac{x_1}{x_2} < \frac{14}{9}.
\end{equation}
Similarly, we have
\begin{equation}
\frac{y_1}{y_2} < \frac{14}{9}.
\end{equation}
Therefore, we obtain
\[
\left( \frac{14}{9} \right)^2 R' > R.
\]
By comparing the corresponding coefficients of each term, we obtain
\[
\frac{14}{9} Q_1' < Q_1 \quad \text{and} \quad \frac{14}{9} Q_2' < Q_2.
\]
From the inequalities above, we conclude that the lemma holds for $a = 2$ as well.

\end{proof}

\begin{lem}\label{ml2}
Let $x, y \in \mathbb{Q}$ satisfy $0 \le x < y \le 1$.
If $x$ and $y$ are of type (2) and $a \in \mathbb{Z}_{>0}$ satisfies $a \ge 1$, then
\[
\mathcal{L}([\phi_{a,a+1}(G(x))]) < \mathcal{L}([\phi_{a,a+1}(G(y))]).
\]
\end{lem}
\begin{proof}
In a similar manner to the previous lemma, we set the notation as in \eqref{aa+1} and \eqref{ba+1}.
If $\mathrm{den}(x) = 1$, then $x = 0$, in which case the lemma clearly holds.
Henceforth, we assume that $0<x<y<1$.
We note that $\mathrm{den}(x) > 1$.
By definition, we have
\[
\mathrm{den}(x) < \mathrm{den}(y) \quad \text{and} \quad \lfloor jx \rfloor = \lfloor jy \rfloor \quad \text{for all } 1 \le j \le \mathrm{den}(x).
\]
Let $n = \mathrm{den}(x)$.
Consequently, we have $G(x,j) = G(y,j)$ for all $1 \le j \le n$.
Since $nx = \mathrm{den}(x)x \in \mathbb{Z}$, it follows that $G(x,n) = 1$.
We have
\[
G(x,-n+1) = \lfloor (-n+1)x \rfloor - \lfloor -nx \rfloor
= \lceil nx \rceil - \lceil (n-1)x \rceil.
\]
Since $nx \in \mathbb{Z}$, we obtain $G(x,-n+1) = 0$.
Since $\mathrm{den}(y) > n$, Lemma~\ref{l2}(7) implies that
\[
G(y,-n+1) = G(y,n) = 1.
\]
Similarly, by Lemma~\ref{l2}(7), we have
\[
G(x,j) = G(x,-j+1) \quad \text{and} \quad G(y,j) = G(y,-j+1)
\]
for all $1 < j < n$.

From the above considerations, we obtain the following properties:

\begin{enumerate}
    \item $G(x,0) = 1$, $G(x,1) = 0$, $G(y,0) = 1$, and $G(y,1) = 0$;
    \item $G(x,n) = 1$, $G(y,n) = 1$, $G(x,-n+1) = 0$, and $G(y,-n+1) = 1$;
    \item $G(x,2) \dots G(x,n-1) = G(x,-1) \dots G(x,-n+2)$;
    \item $G(y,2) \dots G(y,n-1) = G(y,-1) \dots G(y,-n+2)$;
    \item $G(x,2) \dots G(x,n-1) = G(y,2) \dots G(y,n-1)$.
\end{enumerate}
Here, for $n = 2$, the products in the last three items are understood to be the empty word.
Thus, we have
\[
\begin{aligned}
a_1 a_2 a_3 \dots a_{2n} &= a a a_3 \dots a_{2n-2} (a+1)(a+1), \\
b_1 b_2 b_3 \dots b_{2n} &= a a a_3 \dots a_{2n-2} (a+1)(a+1),
\end{aligned}
\]
and
\[
\begin{aligned}
a_{-1} a_{-2} \dots a_{-2n+1} &= (a+1) a_3 \dots a_{2n-2} a a, \\
b_{-1} b_{-2} \dots b_{-2n+1} &= (a+1) a_3 \dots a_{2n-2} (a+1)(a+1).
\end{aligned}
\]
Therefore, we obtain
\begin{align}\label{08291}
\begin{aligned}
[0; a_1, a_2, a_3, \dots, a_{2n}, x_1]
&= [0; a, a, a_3, \dots, a_{2n-2}, a+1, a+1, x_1], \\[1ex]
[0; b_1, b_2, b_3, \dots, b_{2n}, y_1]
&= [0; a, a, a_3, \dots, a_{2n-2}, a+1, a+1, y_1], \\[1ex]
[0; a_{-1}, a_{-2}, \dots, a_{-2n+1}, x_2]
&= [0; a+1, a_3, \dots, a_{2n-2}, a, a, x_2], \\[1ex]
[0; b_{-1}, b_{-2}, \dots, b_{-2n+1}, y_2]
&= [0; a+1, a_3, \dots, a_{2n-2}, a+1, a+1, y_2],
\end{aligned}
\end{align}
where
\[
\begin{aligned}
x_1 &= [a_{2n+1}; a_{2n+2}, \dots], & y_1 &= [b_{2n+1}; b_{2n+2}, \dots], \\
x_2 &= [a_{-2n}; a_{-2n-1}, \dots], & y_2 &= [b_{-2n}; b_{-2n-1}, \dots].
\end{aligned}
\]
Let
\begin{align}\label{m_1m_2}
\begin{pmatrix} m_1 & m_2 \\ m_3 & m_4 \end{pmatrix}
= \begin{pmatrix} 0 & 1 \\ 1 & a_3 \end{pmatrix} \dots \begin{pmatrix} 0 & 1 \\ 1 & a_{2n-2} \end{pmatrix},
\end{align}
where this matrix product is understood to be the identity matrix when $n = 2$.

Therefore, by \eqref{08291}, we have
\[
\begin{aligned}
&\begin{pmatrix} 0 & 1 \\ 1 & a_1 \end{pmatrix} \dots \begin{pmatrix} 0 & 1 \\ 1 & a_{2n} \end{pmatrix} \begin{pmatrix} 1 \\ x_1 \end{pmatrix} \\
&\quad = \begin{pmatrix} 0 & 1 \\ 1 & a \end{pmatrix}^2 \begin{pmatrix} m_1 & m_2 \\ m_3 & m_4 \end{pmatrix} \begin{pmatrix} 0 & 1 \\ 1 & a+1 \end{pmatrix}^2 \begin{pmatrix} 1 \\ x_1 \end{pmatrix} \\
&\quad = \begin{pmatrix} P_1 \\ Q_1 \end{pmatrix},
\end{aligned}
\]
where
\begin{equation}\label{a^2m_4}
\begin{aligned}
P_1 &= (a^3 + 2a^2 + 2a) m_4 x_1 + (a^2 + 2a + 2) m_2 x_1 + (a^2 + a) m_3 x_1 + (a + 1) m_1 x_1 \\
    &\quad + (a^2 + a) m_4 + (a + 1) m_2 + a m_3 + m_1, \\[1.5ex]
Q_1 &= (a^4 + 2a^3 + 3a^2 + 2a + 2) m_4 x_1 + (a^3 + 2a^2 + 2a) m_2 x_1 + (a^3 + a^2 + a + 1) m_3 x_1 + (a^2 + a) m_1 x_1 \\
    &\quad + (a^3 + a^2 + a + 1) m_4 + (a^2 + a) m_2 + (a^2 + 1) m_3 + a m_1.
\end{aligned}
\end{equation}
Therefore, we obtain
\[
[0; a_1, a_2, a_3, \dots, a_{2n}, x_1] = \frac{P_1}{Q_1}.
\]

Similarly,
 by \eqref{08291}, we have
\[
\begin{aligned}
&\begin{pmatrix} 0 & 1 \\ 1 & b_1 \end{pmatrix} \dots \begin{pmatrix} 0 & 1 \\ 1 & b_{2n} \end{pmatrix} \begin{pmatrix} 1 \\ y_1 \end{pmatrix} \\
&\quad = \begin{pmatrix} 0 & 1 \\ 1 & a \end{pmatrix}^2 \begin{pmatrix} m_1 & m_2 \\ m_3 & m_4 \end{pmatrix} \begin{pmatrix} 0 & 1 \\ 1 & a+1 \end{pmatrix}^2 \begin{pmatrix} 1 \\ y_1 \end{pmatrix} \\
&\quad = \begin{pmatrix} P_2 \\ Q_2 \end{pmatrix},
\end{aligned}
\]
where
\begin{equation}\label{a^2m_4}
\begin{aligned}
P_2 &= (a^3 + 2a^2 + 2a) m_4 y_1 + (a^2 + 2a + 2) m_2 y_1 + (a^2 + a) m_3 y_1 + (a + 1) m_1 y_1 \\
    &\quad + (a^2 + a) m_4 + (a + 1) m_2 + a m_3 + m_1, \\[1.5ex]
Q_2 &= (a^4 + 2a^3 + 3a^2 + 2a + 2) m_4 y_1 + (a^3 + 2a^2 + 2a) m_2 y_1 + (a^3 + a^2 + a + 1) m_3 y_1 + (a^2 + a) m_1 y_1 \\
    &\quad + (a^3 + a^2 + a + 1) m_4 + (a^2 + a) m_2 + (a^2 + 1) m_3 + a m_1.
\end{aligned}
\end{equation}
Therefore, we obtain
\[
[0; b_1, b_2, b_3, \dots, b_{2n}, y_1] = \frac{P_2}{Q_2}.
\]
By \eqref{08291}, we have
\[
\begin{aligned}
&\begin{pmatrix} 0 & 1 \\ 1 & a_{-1} \end{pmatrix} \dots \begin{pmatrix} 0 & 1 \\ 1 & a_{-2n+1} \end{pmatrix} \begin{pmatrix} 1 \\ x_2 \end{pmatrix} \\
&\quad = \begin{pmatrix} 0 & 1 \\ 1 & a+1 \end{pmatrix} \begin{pmatrix} m_1 & m_2 \\ m_3 & m_4 \end{pmatrix} \begin{pmatrix} 0 & 1 \\ 1 & a \end{pmatrix}^2 \begin{pmatrix} 1 \\ x_2 \end{pmatrix} \\
&\quad = \begin{pmatrix} P_1' \\ Q_1' \end{pmatrix},
\end{aligned}
\]
where
\begin{equation}\label{m_4x_2}
\begin{aligned}
P_1' &= (a^2 + 1) m_4 x_2 + a m_3 x_2 + a m_4 + m_3, \\[1.5ex]
Q_1' &= (a^3 + a^2 + a + 1) m_4 x_2 + (a^2 + 1) m_2 x_2 + (a^2 + a) m_3 x_2 + a m_1 x_2 \\
     &\quad + (a^2 + a) m_4 + a m_2 + (a + 1) m_3 + m_1.
\end{aligned}
\end{equation}
Therefore, we obtain
\[
[0; a_{-1}, a_{-2}, \dots, a_{-2n+1}, x_2] = \frac{P_1'}{Q_1'}.
\]
By \eqref{08291}, we have
\[
\begin{aligned}
&\begin{pmatrix} 0 & 1 \\ 1 & b_{-1} \end{pmatrix} \dots \begin{pmatrix} 0 & 1 \\ 1 & b_{-2n+1} \end{pmatrix} \begin{pmatrix} 1 \\ y_2 \end{pmatrix} \\
&\quad = \begin{pmatrix} 0 & 1 \\ 1 & a+1 \end{pmatrix} \begin{pmatrix} m_1 & m_2 \\ m_3 & m_4 \end{pmatrix} \begin{pmatrix} 0 & 1 \\ 1 & a+1 \end{pmatrix}^2 \begin{pmatrix} 1 \\ y_2 \end{pmatrix} \\
&\quad = \begin{pmatrix} P_2' \\ Q_2' \end{pmatrix},
\end{aligned}
\]
where
\begin{equation}\label{m_4y_2}
\begin{aligned}
P_2' &= (a^2 + 2a + 2) m_4 y_2 + (a + 1) m_3 y_2 + (a + 1) m_4 + m_3, \\[1.5ex]
Q_2' &= (a^3 + 3a^2 + 4a + 2) m_4 y_2 + (a^2 + 2a + 2) m_2 y_2 + (a^2 + 2a + 1) m_3 y_2 + (a + 1) m_1 y_2 \\
     &\quad + (a^2 + 2a + 1) m_4 + (a + 1) m_2 + (a + 1) m_3 + m_1.
\end{aligned}
\end{equation}
Therefore, we obtain
\[
[0; b_{-1}, b_{-2}, \dots, b_{-2n+1}, y_2] = \frac{P_2'}{Q_2'}.
\]

Therefore, considering $m_1 m_4-m_2 m_3=1$, we obtain
\begin{align}\label{P_2Q_2}
\frac{P_2}{Q_2} - \frac{P_1}{Q_1}
= \frac{x_1-y_1}{Q_1 Q_2}.
\end{align}
Let $R$ denote the numerator $x_1-y_1$.

Similarly, we have
\begin{align}\label{P'_2Q'_2}
\frac{P_2'}{Q_2'} - \frac{P_1'}{Q_1'}
= \frac{(a^2 + a - 1) x_2 y_2 + (a-1)x_2 + (a+2)y_2 + 1}{Q_1' Q_2'}.
\end{align}
Let $R'$ denote the numerator $(a^2 + a - 1) x_2 y_2 + (a-1)x_2 + (a+2)y_2 + 1$.

Let us compare $x_1$ and $y_1$.
Since $x < y$, there exists some $m \in \mathbb{Z}_{>0}$ such that
\[
\lfloor mx \rfloor < \lfloor my \rfloor.
\]
We choose $m$ to be the minimal integer satisfying this property.
Furthermore, since $x$ and $y$ are of type (2), it follows that $m > n$.

Note that
\[
x_1 = [a_{2n+1}; a_{2n+2}, \dots] \quad \text{and} \quad y_1 = [b_{2n+1}; b_{2n+2}, \dots].
\]
Since
\[
\begin{aligned}
a_{2n+1} a_{2n+2} \dots &= \phi_{a,a+1}(G(x,n+1) G(x,n+2) \dots), \\
b_{2n+1} b_{2n+2} \dots &= \phi_{a,a+1}(G(y,n+1) G(y,n+2) \dots),
\end{aligned}
\]
we obtain $x_1 < y_1$.
Moreover, an argument similar to the one in Lemma \ref{ml1} yields
\[
x_1 < x_2 \quad \text{and} \quad y_1 > y_2.
\]

As in the previous lemma, it suffices to show that
\begin{equation}\label{fracP_2'2}
\frac{P_2'}{Q_2'} - \frac{P_1'}{Q_1'} > - \left( \frac{P_2}{Q_2} - \frac{P_1}{Q_1} \right),
\end{equation}
that is,
\begin{equation}\label{0921-1}
\frac{R'}{Q_1'Q_2'} > \frac{|R|}{Q_1Q_2}.
\end{equation}
We now compare $R$ and $R'$.
Assume that $a \ge 2$.
Since $x_1 < y_1$, and
\[
x_1 > [a; a+2] = a + \frac{1}{a+2} \quad \text{and} \quad y_1 < [a+1; a] = a + 1 + \frac{1}{a},
\]
it follows that
\[
|R|=|x_1 - y_1| < a + 1 + \frac{1}{a} - \left( a + \frac{1}{a+2} \right) = 1 + \frac{2}{a(a+2)} < 2.
\]
Since
\[
x_2, y_2 > [a; a+2] = a + \frac{1}{a+2} >a,
\]
we have
\[
\begin{aligned}
R' &= (a^2 + a - 1) x_2 y_2 + (a - 1) x_2 + (a + 2) y_2 + 1 \\
    &> (a^2 + a - 1) a^2 + (a - 1) a + (a + 2) a + 1 \\
   &= a^4 + a^3 + a^2 + a + 1.
\end{aligned}
\]
Consequently, we obtain
\begin{align}\label{0830-1}
\frac{R'}{|R|} > \frac{a^4 + a^3 + a^2 + a + 1}{2}.
\end{align}
Next, let us compare $Q_1$ and $Q_1'$.
As in the previous lemma, we have
\[
x_2 < \left( 1 + \frac{2}{a} \right) x_1.
\]
Therefore, by comparing the corresponding coefficients of each term, we easily obtain
\begin{align}\label{0830-2}
\left( 1 + \frac{2}{a} \right) Q_1 > Q_1'.
\end{align}
As for the comparison between $Q_2$ and $Q_2'$, combining the comparison of corresponding coefficients with $y_1 > y_2$ easily yields
\begin{align}\label{0830-3}
Q_2 > Q_2'.
\end{align}
By taking into account \eqref{0830-1}, \eqref{0830-2}, \eqref{0830-3}, and the inequality
\[
\frac{a^4 + a^3 + a^2 + a + 1}{2} > 1 + \frac{2}{a},
\]
we obtain
\[
\frac{|R|}{Q_1 Q_2} < \frac{R'}{Q_1' Q_2'},
\]
which completes the proof of \eqref{fracP_2'2} with $a\geq 2$.
Now consider the case $a = 1$.
In this case, $R$ and $R'$ are given by
\[
\begin{aligned}
R &=x_1 - y_1,  \\
R' &=x_2y_2 + 3y_2 + 1, 
\end{aligned}
\]
and the denominators $Q_1, Q_1', Q_2, Q_2'$ reduce to
\[
\begin{aligned}
Q_1 &= 2m_1x_1 + 5m_2x_1 + 4m_3x_1 + 10m_4x_1 + m_1 + 2m_2 + 2m_3 + 4m_4, \\
Q_1' &=m_1x_2 + 2m_2x_2 + 2m_3x_2 + 4m_4x_2 + m_1 + m_2 + 2m_3 + 2m_4, \\[1ex]
Q_2 &= 2m_1y_1 + 5m_2y_1 + 4m_3y_1 + 10m_4y_1 + m_1 + 2m_2 + 2m_3 + 4m_4, \\
Q_2' &= 2m_1y_2 + 5m_2y_2 + 4m_3y_2 + 10m_4y_2 + m_1 + 2m_2 + 2m_3 + 4m_4.
\end{aligned}
\]

Since
\[
x_1 > [\overline{1; 2}] = \frac{1 + \sqrt{3}}{2} \quad \text{and} \quad y_1 < [\overline{2; 1}] = 1 + \sqrt{3},
\]
we have
\[
|R| = |x_1 - y_1| < \left| \frac{1 + \sqrt{3}}{2} - (1 + \sqrt{3}) \right| = \frac{1 + \sqrt{3}}{2}.
\]
Since
\[
x_2, y_2 > [1; 3] = \frac{4}{3},
\]
we have
\[
\begin{aligned}
R' &= x_2 y_2 + 3 y_2 + 1 \\
   &> \left( \frac{4}{3} \right)^2 + 3 \left( \frac{4}{3} \right) + 1 = \frac{61}{9}.
\end{aligned}
\]
\begin{align}\label{0830-4}
\frac{R'}{|R|} > \frac{122}{9(1 + \sqrt{3})}.
\end{align}
Regarding the comparison between $Q_1$ and $Q_1'$, since
\[
\frac{x_2}{x_1} < \frac{1 + \sqrt{3}}{\frac{1 + \sqrt{3}}{2}} = 2,
\]
comparing the corresponding coefficients of each term yields
\[
2 Q_1 > Q_1'.
\]

As for the comparison between $Q_2$ and $Q_2'$, an argument similar to the one above gives
\[
Q_2 > Q_2'.
\]
Since
\[
\frac{122}{9(1 + \sqrt{3})} > 2,
\]
we obtain
\[
\frac{|R|}{Q_1 Q_2} < \frac{R'}{Q_1' Q_2'},
\]
which completes the proof of \eqref{fracP_2'2}.

\end{proof}

\begin{lem}\label{ml3}
Let $x, y \in \mathbb{Q}$ satisfy $0 \le x < y \le 1$.
If $x$ and $y$ are of type (3) and $a \in \mathbb{Z}_{>0}$ satisfies $a \ge 2$, then
\[
\mathcal{L}([\phi_{a,a+1}(G(x))]) > \mathcal{L}([\phi_{a,a+1}(G(y))]).
\]
\end{lem}

\begin{proof}
In a similar manner to the previous lemma, we set the notation as in \eqref{aa+1} and \eqref{ba+1}.
If $\mathrm{den}(y) = 1$, then $y = 1$, in which case the lemma clearly holds.
Henceforth, we assume that $0<x<y<1$.
We note that $\mathrm{den}(y) > 1$.
By definition,
we have
\[
\begin{gathered}
\mathrm{den}(x) > \mathrm{den}(y), \\
\lfloor jx \rfloor = \lfloor jy \rfloor \quad \text{for all } 1 \le j < \mathrm{den}(y), \quad \text{and} \\
\lfloor jx \rfloor < \lfloor jy \rfloor \quad \text{for } j = \mathrm{den}(y).
\end{gathered}
\]
Let $n=\mathrm{den}(y)$.
We have $G(x,j) = G(y,j)$ for all $1 \le j < n$.
Since $ny = \mathrm{den}(y)y \in \mathbb{Z}$, it follows that $G(y,n) = 1$.
On the other hand, this fact ($ny = \mathrm{den}(y)y \in \mathbb{Z}$) together with
\[
\lfloor (n-1)x \rfloor = \lfloor (n-1)y \rfloor \quad \text{and} \quad \lfloor nx \rfloor < \lfloor ny \rfloor
\]
implies that $G(x,n) = 0$.

Since $n = \mathrm{den}(y) < \mathrm{den}(x)$,  by Lemma \ref{l2}(7) we have
\[
G(x,j) = G(x,-j+1) \quad \text{and} \quad G(y,j) = G(y,-j+1)
\]
for all $1 < j < n$.
Furthermore, since $ny \in \mathbb{Z}$, we get $G(y,-n+1) = 0$.
On the other hand,
\[
G(x,n) = G(x,-n+1) = 0.
\]
From the above considerations, we obtain the following properties:

\begin{enumerate}
    \item $G(x,0) = 1$, $G(x,1) = 1$, $G(y,0) = 1$, and $G(y,1) = 0$;
    \item $G(x,n) = 0$, $G(y,n) = 1$, $G(x,-n+1) = 0$, and $G(y,-n+1) = 0$;
    \item $G(x,2) \dots G(x,n-1) = G(x,-1) \dots G(x,-n+2)$;
    \item $G(y,2) \dots G(y,n-1) = G(y,-1) \dots G(y,-n+2)$;
    \item $G(x,2) \dots G(x,n-1) = G(y,2) \dots G(y,n-1)$.
\end{enumerate}
Here, for $n = 2$, the products in the last three items are understood to be the empty word.
Thus, we have
\[
\begin{aligned}
a_1 a_2 a_3 \dots a_{2n} &= a a a_3 \dots a_{2n-2} a a, \\
b_1 b_2 b_3 \dots b_{2n} &= a a a_3 \dots a_{2n-2} (a+1)(a+1),
\end{aligned}
\]
and
\[
\begin{aligned}
a_{-1} a_{-2} \dots a_{-2n+1} &= (a+1) a_3 \dots a_{2n-2} a a, \\
b_{-1} b_{-2} \dots b_{-2n+1} &= (a+1) a_3 \dots a_{2n-2} a a.
\end{aligned}
\]
Therefore, we obtain
\begin{align}\label{09091}
\begin{aligned}
[0; a_1, a_2, a_3, \dots, a_{2n}, x_1]
&= [0; a, a, a_3, \dots, a_{2n-2}, a, a, x_1], \\[1ex]
[0; b_1, b_2, b_3, \dots, b_{2n}, y_1]
&= [0; a, a, a_3, \dots, a_{2n-2}, a+1, a+1, y_1], \\[1ex]
[0; a_{-1}, a_{-2}, \dots, a_{-2n+1}, x_2]
&= [0; a+1, a_3, \dots, a_{2n-2}, a, a, x_2], \\[1ex]
[0; b_{-1}, b_{-2}, \dots, b_{-2n+1}, y_2]
&= [0; a+1, a_3, \dots, a_{2n-2}, a, a, y_2],
\end{aligned}
\end{align}
where
\[
\begin{aligned}
x_1 &= [a_{2n+1}; a_{2n+2}, \dots], & y_1 &= [b_{2n+1}; b_{2n+2}, \dots], \\
x_2 &= [a_{-2n}; a_{-2n-1}, \dots], & y_2 &= [b_{-2n}; b_{-2n-1}, \dots].
\end{aligned}
\]
Let
\begin{align}\label{m_1m_2}
\begin{pmatrix} m_1 & m_2 \\ m_3 & m_4 \end{pmatrix}
= \begin{pmatrix} 0 & 1 \\ 1 & a_3 \end{pmatrix} \dots \begin{pmatrix} 0 & 1 \\ 1 & a_{2n-2} \end{pmatrix},
\end{align}
where this matrix product is understood to be the identity matrix when $n = 2$.
Therefore, by \eqref{09091}, we have
\[
\begin{aligned}
&\begin{pmatrix} 0 & 1 \\ 1 & a_1 \end{pmatrix} \dots \begin{pmatrix} 0 & 1 \\ 1 & a_{2n} \end{pmatrix} \begin{pmatrix} 1 \\ x_1 \end{pmatrix} \\
&\quad = \begin{pmatrix} 0 & 1 \\ 1 & a \end{pmatrix}^2 \begin{pmatrix} m_1 & m_2 \\ m_3 & m_4 \end{pmatrix} \begin{pmatrix} 0 & 1 \\ 1 & a \end{pmatrix}^2 \begin{pmatrix} 1 \\ x_1 \end{pmatrix} \\
&\quad = \begin{pmatrix} P_1 \\ Q_1 \end{pmatrix},
\end{aligned}
\]
where
\begin{equation}\label{09092}
\begin{aligned}
P_1 &= a^3 m_4x_1 + a^2 (m_2 + m_3)x_1 + a m_1x_1 + a m_4x_1 + m_2x_1 \\
    &\quad + a^2m_4 + a(m_2 + m_3) + m_1, \\[1.5ex]
Q_1 &= (a^4 + 2a^2 + 1)m_4 x_1 + (a^3 + a)(m_2 + m_3)x_1 + a^2 m_1x_1 \\
    &\quad + (a^3+a)m_4 + a^2m_2 + (a^2+1)m_3 + am_1.
\end{aligned}
\end{equation}
Therefore, we obtain
\[
[0; a_1, a_2, a_3, \dots, a_{2n}, x_1] = \frac{P_1}{Q_1}.
\]
Similarly,
 by \eqref{09091}, we have
\[
\begin{aligned}
&\begin{pmatrix} 0 & 1 \\ 1 & b_1 \end{pmatrix} \dots \begin{pmatrix} 0 & 1 \\ 1 & b_{2n} \end{pmatrix} \begin{pmatrix} 1 \\ y_1 \end{pmatrix} \\
&\quad = \begin{pmatrix} 0 & 1 \\ 1 & a \end{pmatrix}^2 \begin{pmatrix} m_1 & m_2 \\ m_3 & m_4 \end{pmatrix} \begin{pmatrix} 0 & 1 \\ 1 & a+1 \end{pmatrix}^2 \begin{pmatrix} 1 \\ y_1 \end{pmatrix} \\
&\quad = \begin{pmatrix} P_2 \\ Q_2 \end{pmatrix},
\end{aligned}
\]
where
\begin{equation}\label{a^2m_4-2}
\begin{aligned}
P_2 &= (a^3 + 2a^2 + 2a) m_4 y_1 + (a^2 + 2a + 2) m_2 y_1 + (a^2 + a) m_3 y_1 + (a + 1) m_1 y_1 \\
    &\quad + (a^2 + a) m_4 + (a + 1) m_2 + a m_3 + m_1, \\[1.5ex]
Q_2 &= (a^4 + 2a^3 + 3a^2 + 2a + 2) m_4 y_1 + (a^3 + 2a^2 + 2a) m_2 y_1 + (a^3 + a^2 + a + 1) m_3 y_1 + (a^2 + a) m_1 y_1 \\
    &\quad + (a^3 + a^2 + a + 1) m_4 + (a^2 + a) m_2 + (a^2 + 1) m_3 + a m_1.
\end{aligned}
\end{equation}
Therefore, we obtain
\[
[0; b_1, b_2, b_3, \dots, b_{2n}, y_1] = \frac{P_2}{Q_2}.
\]
By \eqref{09091}, we have
\[
\begin{aligned}
&\begin{pmatrix} 0 & 1 \\ 1 & a_{-1} \end{pmatrix} \dots \begin{pmatrix} 0 & 1 \\ 1 & a_{-2n+1} \end{pmatrix} \begin{pmatrix} 1 \\ x_2 \end{pmatrix} \\
&\quad = \begin{pmatrix} 0 & 1 \\ 1 & a+1 \end{pmatrix} \begin{pmatrix} m_1 & m_2 \\ m_3 & m_4 \end{pmatrix} \begin{pmatrix} 0 & 1 \\ 1 & a \end{pmatrix}^2 \begin{pmatrix} 1 \\ x_2 \end{pmatrix} \\
&\quad = \begin{pmatrix} P_1' \\ Q_1' \end{pmatrix},
\end{aligned}
\]
where
\begin{equation}\label{m_4x_2-2}
\begin{aligned}
P_1' &= (a^2 + 1) m_4 x_2 + a m_3 x_2 + a m_4 + m_3, \\[1.5ex]
Q_1' &= (a^3 + a^2 + a + 1) m_4 x_2 + (a^2 + 1) m_2 x_2 + (a^2 + a) m_3 x_2 + a m_1 x_2 \\
     &\quad + (a^2 + a) m_4 + a m_2 + (a + 1) m_3 + m_1.
\end{aligned}
\end{equation}
Therefore, we obtain
\[
[0; a_{-1}, a_{-2}, \dots, a_{-2n+1}, x_2] = \frac{P_1'}{Q_1'}.
\]

By \eqref{09091}, we have
\[
\begin{aligned}
&\begin{pmatrix} 0 & 1 \\ 1 & b_{-1} \end{pmatrix} \dots \begin{pmatrix} 0 & 1 \\ 1 & b_{-2n+1} \end{pmatrix} \begin{pmatrix} 1 \\ y_2 \end{pmatrix} \\
&\quad = \begin{pmatrix} 0 & 1 \\ 1 & a+1 \end{pmatrix} \begin{pmatrix} m_1 & m_2 \\ m_3 & m_4 \end{pmatrix} \begin{pmatrix} 0 & 1 \\ 1 & a \end{pmatrix}^2 \begin{pmatrix} 1 \\ y_2 \end{pmatrix} \\
&\quad = \begin{pmatrix} P_2' \\ Q_2' \end{pmatrix},
\end{aligned}
\]
where
\begin{equation}\label{m_4y_2-2}
\begin{aligned}
P_2' &= (a^2 + 1) m_4 y_2 + a m_3 y_2 + a m_4 + m_3, \\[1.5ex]
Q_2' &= (a^3 + a^2 + a + 1) m_4 y_2 + (a^2 + 1) m_2 y_2 + (a^2 + a) m_3 y_2 + a m_1 y_2 \\
     &\quad + (a^2 + a) m_4 + a m_2 + (a + 1) m_3 + m_1.
\end{aligned}
\end{equation}
Therefore, we obtain
\[
[0; b_{-1}, b_{-2}, \dots, b_{-2n+1}, y_2] = \frac{P_2'}{Q_2'}.
\]

By taking $m_1 m_4 - m_2 m_3 = 1$ into account, we obtain
\begin{equation}\label{P_2Q_2-3}
\frac{P_2}{Q_2} - \frac{P_1}{Q_1}
= -\frac{(a^2 + a - 1) x_1 y_1 + (a-1)x_1 + (a+2)y_1 + 1}{Q_1 Q_2}
\end{equation}
and
\[
\frac{P_2'}{Q_2'} - \frac{P_1'}{Q_1'}
= -\frac{x_2 - y_2}{Q_1' Q_2'}.
\]
Let $R$ denote the numerator $(a^2 + a - 1) x_1 y_1 + (a-1)x_1 + (a+2)y_1 + 1$, 
and let $R'$ denote the numerator $x_2 - y_2$.
Similarly to Lemma \ref{ml1}, we have
\[
x_1 > x_2 \quad \text{and} \quad y_1 < y_2.
\]
The comparison between $x_2$ and $y_2$ is not straightforward.

To prove the lemma, it suffices to show that
\begin{equation}\label{fracP_2'-3}
\frac{P_2'}{Q_2'} - \frac{P_1'}{Q_1'} < - \left( \frac{P_2}{Q_2} - \frac{P_1}{Q_1} \right),
\end{equation}
which is ensured by the stronger inequality
\begin{equation}\label{0921-3}
\frac{|R'|}{Q_1'Q_2'} < \frac{R}{Q_1Q_2}.
\end{equation}
We now compare $R$ and $|R'|$.
Assume that $a \ge 3$.
As shown in the previous lemma, we have
\begin{align}\label{0911-1}
a + \frac{1}{a+2} < x_i, y_i < a + 1 + \frac{1}{a}
\end{align}
for $i = 1, 2$.
Therefore,
\[
\begin{aligned}
R &= (a^2 + a - 1) x_1 y_1 + (a - 1) x_1 + (a + 2) y_1 + 1 \\
  &> (a^2 + a - 1) a^2 + (a - 1) a + (a + 2) a + 1 \\
  &= a^4 + a^3 + a^2 + a + 1
\end{aligned}
\]
and
\[
|R'| = |x_2 - y_2| < a + 1 + \frac{1}{a} - \left( a + \frac{1}{a+2} \right) = 1 + \frac{2}{a(a+2)} < 2.
\]
Therefore, we obtain
\[
\frac{R}{|R'|} > \frac{a^4 + a^3 + a^2 + a + 1}{2}.
\]

We now compare $Q_1$ and $Q_1'$.
As in the previous lemma, we have
\[
\left( 1 + \frac{2}{a} \right) x_2 > x_1.
\]
Thus, by comparing the corresponding coefficients, we obtain
\[
a \left( 1 + \frac{2}{a} \right) Q_1' > Q_1.
\]
Next, we compare $Q_2$ and $Q_2'$.
Taking $y_2 > y_1$ into account, a comparison of the corresponding coefficients yields
\[
\left( a + 2 + \frac{1}{a} \right) Q_2' > Q_2.
\]
Consequently, we have
\[
\frac{Q_1 Q_2}{Q_1' Q_2'} < \left( a + 2 + \frac{1}{a} \right) a \left( 1 + \frac{2}{a} \right).
\]
Since it is straightforward to verify that
\[
\left( a + 2 + \frac{1}{a} \right) a \left( 1 + \frac{2}{a} \right) < \frac{a^4 + a^3 + a^2 + a + 1}{2},
\]
we arrive at \eqref{fracP_2'-3}.
Now consider the case $a = 2$.
In this case, $R$ and $R'$ are given by
\[
\begin{aligned}
R &= 5 x_1 y_1 + x_1 + 4 y_1 + 1, \\
R' &= x_2 - y_2,
\end{aligned}
\]
and the denominators $Q_1, Q_1', Q_2, Q_2'$ reduce to
\[
\begin{aligned}
Q_1 &= 4 m_1 x_1 + 10 m_2 x_1 + 10 m_3 x_1 + 25 m_4 x_1 + 2 m_1 + 4 m_2 + 5 m_3 + 10 m_4, \\
Q_1' &= 2 m_1 x_2 + 5 m_2 x_2 + 6 m_3 x_2 + 15 m_4 x_2 + m_1 + 2 m_2 + 3 m_3 + 6 m_4, \\[1ex]
Q_2 &= 6 m_1 y_1 + 20 m_2 y_1 + 15 m_3 y_1 + 50 m_4 y_1 + 2 m_1 + 6 m_2 + 5 m_3 + 15 m_4, \\
Q_2' &= 2 m_1 y_2 + 5 m_2 y_2 + 6 m_3 y_2 + 15 m_4 y_2 + m_1 + 2 m_2 + 3 m_3 + 6 m_4.
\end{aligned}
\]
Similarly, we have
\[
\frac{9}{4} < x_i, y_i < \frac{7}{2}
\]
for $i = 1, 2$.
Therefore, we have
\[
R > 5 \left( \frac{9}{4} \right)^2 + \frac{9}{4} + 4 \cdot \frac{9}{4} + 1 = \frac{601}{16}
\]
and
\[
|R'| = |x_2 - y_2| < \frac{7}{2} - \frac{9}{4} = \frac{5}{4}.
\]
Combining these bounds, we obtain
\[
\frac{R}{|R'|} > \frac{601/16}{5/4} = \frac{601}{20}.
\]
Similarly, we have
\[
\frac{x_1}{x_2} < \frac{14}{9}.
\]
Consequently, we obtain
\[
\frac{28}{9} Q_1' > Q_1 \quad \text{and} \quad  4Q_2' > Q_2,
\]
which leads to
\[
\frac{Q_1 Q_2}{Q_1' Q_2'} < \frac{112}{9}.
\]
Since $\frac{601}{20} > \frac{112}{9}$, combining these bounds yields
\[
\frac{R}{Q_1 Q_2} > \frac{|R'|}{Q_1' Q_2'},
\]
which completes the proof of \eqref{fracP_2'-3}.
\end{proof}

\section{Conjecture}\label{sec:conjecture}
It is a natural step to consider the Lagrange constants of $\phi_{a,b}(G(x))$ for general $b > a$. 
Based on  numerical computations, we formulate the following conjecture.
Interestingly, the numerical data suggest a change in the ordering behavior of type (1) at the threshold \(b=a^2+2\).
To investigate the conjectured ordering relations, we carried out randomized numerical experiments. In each trial, we selected two distinct rational numbers \(x,y\in(0,1)\), represented in reduced form with denominators at most \(14\), and arranged them so that \(x<y\). We then classified the pair according to the three types defined in Lemma~\ref{l1}.
We also selected integer parameters \(a\) and \(b\), with \(1\le a\le 11\) and \(b>a\), using a randomized procedure designed to sample parameter pairs below, at, and above the proposed threshold \(b=a^2+2\). For each trial, we numerically evaluated the corresponding Lagrange constants and checked whether their ordering agreed with the conjecture. We performed \(100{,}000\) randomized trials. All trials agreed with the predicted ordering relations, and no counterexamples were found.

\begin{con}
Let $a, b \in \mathbb{Z}_{>0}$ satisfy $1 \le a < b$, and let $x, y \in \mathbb{Q} \cap [0, 1]$ with $x < y$.

\begin{enumerate}
\item[\rm (I)] Suppose that $b < a^2 + 2$. Then the following hold:
\begin{enumerate}
\item[\rm (1)] If $(x, y)$ is of type (1) and $a \ge 2$, then
\[
\mathcal{L}([\phi_{a,b}(G(x))]) < \mathcal{L}([\phi_{a,b}(G(y))]).
\]
\item[\rm (2)] If $(x, y)$ is of type (2), then
\[
\mathcal{L}([\phi_{a,b}(G(x))]) < \mathcal{L}([\phi_{a,b}(G(y))]).
\]
\item[\rm (3)] If $(x, y)$ is of type (3) and $a \ge 2$, then
\[
\mathcal{L}([\phi_{a,b}(G(x))]) > \mathcal{L}([\phi_{a,b}(G(y))]).
\]
\end{enumerate}

\item[\rm (II)] Suppose that $b \ge a^2 + 2$. Then the following hold:
\begin{enumerate}
\item[\rm (1)] If $(x, y)$ is of type (1), then
\[
\mathcal{L}([\phi_{a,b}(G(x))]) > \mathcal{L}([\phi_{a,b}(G(y))]).
\]
\item[\rm (2)] If $(x, y)$ is of type (2), then
\[
\mathcal{L}([\phi_{a,b}(G(x))]) < \mathcal{L}([\phi_{a,b}(G(y))]).
\]
\item[\rm (3)] If $(x, y)$ is of type (3), then
\[
\mathcal{L}([\phi_{a,b}(G(x))]) > \mathcal{L}([\phi_{a,b}(G(y))]).
\]
\end{enumerate}
\end{enumerate}
\end{con}

\section*{Declaration of Generative AI and AI-assisted Technologies in the Manuscript Preparation Process}
During the preparation of this work, the author used ChatGPT (OpenAI) and Gemini (Google) to assist with English language polishing and LaTeX formatting. After using these tools, the author reviewed and edited the content as necessary and takes full responsibility for the content of the publication.

\section*{Acknowledgements}
The author would like to express his sincere gratitude to Professor Ikuro Sato for inspiring this study on generalizations of the Lagrange spectrum. The author is also deeply grateful to Professor Junya Satoh for providing excellent research facilities and support.
This research did not receive any specific grant from funding agencies in the public, commercial, or not-for-profit sectors.

\vspace{2cm}

\noindent
Shin-ichi Yasutomi: Graduate School of Informatics, Nagoya  University, JAPAN\\
{\it E-mail address: yasutomi.shinichi.f1@f.mail.nagoya-u.ac.jp}

\begin{thebibliography}{9}

\bibitem{Aigner2013}
M.~Aigner,
\emph{Markov's Theorem and 100 Years of the Uniqueness Conjecture},
Springer Undergraduate Mathematics Series, Springer, Cham, 2013.
https://doi.org/10.1007/978-3-319-00888-2

\bibitem{Banaian2026}
E.~Banaian,
\emph{Orderings on $k$-Markov numbers},
Ramanujan J. \textbf{71}(2026), 12.
https://doi.org/10.1007/s11139-026-01468-5
\bibitem{Gyoda2026}
Y.~Gyoda,
\emph{Generalized discrete Markov spectra},
preprint, arXiv:2512.04547, 2025.
https://doi.org/10.48550/arXiv.2512.04547

\bibitem{Lee2023}
K.~Lee, L.~Li, M.~Rabideau, and R.~Schiffler,
\emph{On the ordering of the Markov numbers},
Adv. Appl. Math. \textbf{143} (2023), 102453.
https://doi.org/10.1016/j.aam.2022.102453

\bibitem{Lothaire2002}
M.~Lothaire,
\emph{Algebraic Combinatorics on Words},
Cambridge University Press, Cambridge, 2002.
https://doi.org/10.1017/CBO9781107326019

\bibitem{Markoff1879}
A.~Markoff,
\emph{Sur les formes quadratiques binaires ind{\'e}finies},
Math. Ann. \textbf{15} (1879), 381--406.
https://doi.org/10.1007/BF02086269

\bibitem{Markoff1880}
A.~Markoff,
\emph{Sur les formes quadratiques binaires ind{\'e}finies (second m{\'e}moire)},
Math. Ann. \textbf{17} (1880), 379--399.
https://doi.org/10.1007/BF01446234

\bibitem{Rabideau2020}
M.~Rabideau and R.~Schiffler,
\emph{Continued fractions and orderings on the Markov numbers},
Adv. Math. \textbf{370} (2020), 107231.
https://doi.org/10.1016/j.aim.2020.107231

\bibitem{Reutenauer2019}
C.~Reutenauer,
\emph{From Christoffel Words to Markoff Numbers},
Oxford University Press, Oxford, 2019.
https://doi.org/10.1093/9780197907535.001.0001

\end{thebibliography}
\end{document}